\documentclass[english,12pt]{smfart}
\usepackage{amsfonts}
\usepackage{amscd}
\usepackage{graphicx}
\usepackage[all]{xypic}
\usepackage[all]{xy}
\usepackage[francais]{babel}
\usepackage[latin1]{inputenc}

\CompileMatrices

\swapnumbers

\def\Var{{\rm Var}}

\let\cal\mathcal

\def\11{{\mathbf 1}}
\def\AA{{\mathbf A}}

\mathchardef\alphag="7C0B \mathchardef\betag="7C0C
\mathchardef\gammag="7C0D \mathchardef\deltag="7C0E
\mathchardef\varepsilong="7C22 \mathchardef\varphig="7C27
\mathchardef\psig="7C20 \mathchardef\zetag="7C10
\mathchardef\epsilong="7C0F \mathchardef\rhog="7C1A
\mathchardef\taug="7C1C \mathchardef\upsilong="7C1D
\mathchardef\iotag="7C13 \mathchardef\thetag="7C12
\mathchardef\pig="7C19 \mathchardef\sigmag="7C1B
\mathchardef\etag="7C11 \mathchardef\omegag="7C21
\mathchardef\kappag="7C14 \mathchardef\lambdag="7C15
\mathchardef\mug="7C16 \mathchardef\xig="7C18
\mathchardef\chig="7C1F \mathchardef\nug="7C17
\mathchardef\varthetag="7C23 \mathchardef\varpig="7C24
\mathchardef\varrhog="7C25 \mathchardef\varsigmag="7C26
\mathchardef\Omegag="7C0A \mathchardef\Thetag="7C02
\mathchardef\Sigmag="7C06 \mathchardef\Deltag="7C01
\mathchardef\Phig="7C08 \mathchardef\Gammag="7C00
\mathchardef\Psig="7C09 \mathchardef\Lambdag="7C03
\mathchardef\Xig="7C04 \mathchardef\Pig="7C05
\mathchardef\Upsilong="7C07

\newtheorem{ex}[subsubsection]{Example}
\newtheorem{thm}[subsubsection]{Theorem}
\newtheorem{theorem}[subsection]{Theorem}
\newtheorem{lem}[subsubsection]{Lemma}
\newtheorem{lemma}[subsubsection]{Lemma}

\newtheorem{prop}[subsubsection]{Proposition}

\theoremstyle{definition}
\newtheorem{definition}[subsubsection]{Definition}
\newtheorem{example}[subsection]{Example}

\newtheorem{def-prop}[subsection]{Proposition-Definition}
\newtheorem{def-theorem}[subsection]{Theorem-Definition}
\newtheorem{def-lem}[subsection]{Lemma-Definition}

\theoremstyle{remark}
\newtheorem{remark}[subsubsection]{Remark}

\newtheorem{r and comp}[subsection]{Remarks and computations}
\newtheorem{notation}[subsubsection]{Notation}

\theoremstyle{plain}

\numberwithin{equation}{subsection}

\def\boxit#1#2{\setbox1=\hbox{\kern#1{#2}\kern#1}%
\dimen1=\ht1 \advance\dimen1 by #1 \dimen2=\dp1 \advance\dimen2 by
#1
\setbox1=\hbox{\vrule height\dimen1 depth\dimen2\box1\vrule}%
\setbox1=\vbox{\hrule\box1\hrule}%
\advance\dimen1 by .4pt \ht1=\dimen1 \advance\dimen2 by .4pt
\dp1=\dimen2 \box1\relax}

\renewcommand{\theequation}{\thesubsection.\arabic{equation}}

\let\cal\mathcal

\def\AA{{\mathbf A}}

\mathchardef\alphag="7C0B \mathchardef\betag="7C0C
\mathchardef\gammag="7C0D \mathchardef\deltag="7C0E
\mathchardef\varepsilong="7C22 \mathchardef\varphig="7C27
\mathchardef\psig="7C20 \mathchardef\zetag="7C10
\mathchardef\epsilong="7C0F \mathchardef\rhog="7C1A
\mathchardef\taug="7C1C \mathchardef\upsilong="7C1D
\mathchardef\iotag="7C13 \mathchardef\thetag="7C12
\mathchardef\pig="7C19 \mathchardef\sigmag="7C1B
\mathchardef\etag="7C11 \mathchardef\omegag="7C21
\mathchardef\kappag="7C14 \mathchardef\lambdag="7C15
\mathchardef\mug="7C16 \mathchardef\xig="7C18
\mathchardef\chig="7C1F \mathchardef\nug="7C17
\mathchardef\varthetag="7C23 \mathchardef\varpig="7C24
\mathchardef\varrhog="7C25 \mathchardef\varsigmag="7C26
\mathchardef\Omegag="7C0A \mathchardef\Thetag="7C02
\mathchardef\Sigmag="7C06 \mathchardef\Deltag="7C01
\mathchardef\Phig="7C08 \mathchardef\Gammag="7C00
\mathchardef\Psig="7C09 \mathchardef\Lambdag="7C03
\mathchardef\Xig="7C04 \mathchardef\Pig="7C05
\mathchardef\Upsilong="7C07

\def\R{{ \mathbb R}}
\def\Z{{ \mathbb Z}}
\def\L{{ \mathbb L}}
\def\C{{ \mathbb C}}
\def\P{{ \mathbb P}}
\def\N{{ \mathbb N}}

\newcommand{\eps}{\varepsilon}

\newcommand{\A}{\mathbb A}

\DeclareMathOperator{\jac}{jac}

\DeclareMathOperator{\grad}{grad}

\date{\today}
\begin{document}

\title[\tiny{Grothendieck ring of semialgebraic formulas  }]{\rm Grothendieck ring of semialgebraic formulas and motivic real Milnor fibres}{}
\author{Georges COMTE}
\address{Laboratoire de Math\'ematiques de l'Universit\'e de Savoie, UMR CNRS 5127,
B\^atiment Chablais, Campus scientifique, 
73376 Le Bourget-du-Lac cedex, France}
\email{georges.comte@univ-savoie.fr}
\urladdr{http://gc83.perso.sfr.fr/}

\author{Goulwen FICHOU}
\address{IRMAR, UMR 6625 du CNRS, Campus de Beaulieu, 35042 Rennes cedex, France}
\email{goulwen.fichou@univ-rennes1.fr}
\urladdr{http://perso.univ-rennes1.fr/goulwen.fichou/}

\begin{abstract}
 We define a Grothendieck ring for basic
real semialgebraic formulas, 
that is for systems of real algebraic equations and inequalities.
In this ring the class of a formula takes into consideration 
the algebraic nature of the set of points satisfying this formula and this ring 
contains as a subring the usual Grothendieck ring of real algebraic formulas.
We give a realization of our ring that allows us to express a class as a 
$\Z[\frac{1}{2}]$-linear combination of 
classes of real algebraic formulas, so this realization gives rise to a notion of 
virtual Poincar\'e polynomial for basic semialgebraic
formulas. We then define zeta functions with coefficients in our ring, built on 
semialgebraic 
formulas in arc spaces. We show that they are rational 
and relate them to the topology of real Milnor fibres.  
\end{abstract}

\maketitle

\renewcommand{\partname}{}

\section*{Introduction}

Let us consider the category $SA(\R)$ of real semialgebraic sets, 
the morphisms being the semialgebraic maps. We denote by 
$(K_0(SA(\R)),+,\cdot)$, or simply $K_0(SA(\R))$,
the Grothendieck ring of $SA(\R)$, that is to say the free ring  
generated by all semialgebraic sets A, denoted by $[A]$ as viewed as
element of 
$K_0(SA(\R))$, in such a way that for all objects $A,B$ of $SA(\R)$ 
one has:
$[A\times B]=[A]\cdot[B]$ and for all closed semialgebraic set $F$ in $A$
one has: $[A\setminus F]+[F]=[A]$ (this implies that for every
semialgebraic sets $A,B$, one has: $[A\cup B]=[A]+[B]-[A\cap B]$).

When  furthermore an equivalence relation for semialgebraic
sets is previously considered for the definition of $K_0(SA(\R))$, 
one has to be aware that the induced quotient ring, 
still denoted for simplicity by  $K_0(SA(\R))$, may dramatically collapse.
For instance let us consider the equivalence relation $A\sim B$ if and only if
there exists a semialgebraic bijection from $A$ to $B$.
In this case we simply say that $A$ and $B$ are isomorphic.
 Then for the definition of $K_0(SA(\R))$, starting from 
classes of isomorphic sets instead of simply sets, one obtains a quite
trivial Grothendieck ring, namely $K_0(SA(\R))=\Z$. 
Indeed,  denoting
$[\R]$ by $\L$ and $[\{*\}]$ by $\P$, from the fact that 
$\{*\}\times \{*\}\sim \{*\}$, one gets 
$$ \P^k=\P, \ \forall k\in \N^*,$$
and from the fact that $\R=
]-\infty, 0[ \cup \{0\} \cup ]0,+\infty[$ and that intervals
of the same type are isomorphic, one gets 
$$ \L=-\P. $$ 

On the other hand, by the semialgebraic cell decomposition theorem,
we  obtain that a real semialgebraic set is a finite union of disjoint
open cells, each of which is isomorphic to $\R^k$, with $k\in \N$
(with the convention that $\R^0=\{*\}$). It follows that 
$K_0(SA(\R))=<\P>$, the ring generated by $\P$. 
At this point, the ring $<\P>$ could be trivial.
But one knows that the Euler-Poincar\'e characteristic with compact 
supports $\chi_c:SA(\R)\to \Z$  is surjective. Let us recall that the Euler-Poincar\'e characteristic with compact 
supports is a topological invariant defined on locally compact semialgebraic sets and uniquely extended  to an additive 
invariant on all semialgebraic sets (see for instance \cite{Coste}, Theorem 1.22). Since $\chi_c$ is additive, 
multiplicative and invariant under isomorphims, it factors through
$K_0(SA(\R))$, giving a surjective morphism of rings, and finally
an isomorphism of rings, still denoted for simplicity by 
$\chi_c$ (cf also \cite{Q}):
\vskip0mm
 $$ \shorthandoff{;:!?}
\xymatrix{   
SA(\R)\ar[d]  \ar[r]^{\chi_c}&  \Z       \\
  <\P> =K_0(SA(\R)) \ar[ru]_{\chi_c}    &     \\}  $$

The characteristic $\chi_c(A)$ of a semialgebraic set $A$ is in fact
defined in the same way, so we obtain the equality 
$K_0(SA(\R))=<\P>$,
that is from a specific cell decomposition of $A$, where $<\P>$ is 
replaced by $\chi_c(\{*\})=1$. The difficulty in the definition of $\chi_c$ is 
then to show that $\chi_c$ is 
independent of the choice of the  cell decomposition of $A$ (it technically 
consists in showing that the definition of  $\chi_c(A)$ does not depend on the 
isomorphism class of $A$, see \cite{Dri} for instance).

When one starts from the category of real algebraic varieties $\Var_\R$ or from
the category of
real algebraic sets $\R \Var$, 
as we do not have algebraic cell decompositions, we could expect 
that the induced Grothendieck ring $K_0(\Var_\R)$ 
is no longer trivial. This is indeed the case, since for instance
the virtual Poincar\'e polynomial morphism  
factors through $K_0(\Var_\R)$ and has image $\Z[u]$ (see \cite{MCP}).
 
 The first part of this article is devoted  to the construction of  
non-trivial 
Gro\-then\-dieck ring $K_0(BSA_{\R})$ associated to $SA(\R)$, with a  
canonical inclusion
$$K_0(\Var_\R)  \hookrightarrow K_0(BSA_{\R}),$$  
that gives rise to a notion of virtual 
Poincar\'e polynomial
for basic real semialgebraic formulas extending the virtual Poincar\'e polynomial
of real algebraic sets
and that allows factorization of the Euler-Poincar\'e
characteristic of real semialgebraic sets of points satisfying the formulas. 

To be more precise, we first construct  
$K_0(BSA_{\R}))$, the Grothendieck ring of basic real semialgebraic 
formulas (which are quantifier free real semialgebraic formulas
or simply systems of real algebraic equations and inequalities) where the class
of basic formulas without inequality is considered up to algebraic isomorphism
of the underlying real algebraic varieties. 
In general a class in $K_0(BSA_{\R})$ 
of a basic real semialgebraic 
formula  depends strongly on the formula itself rather than 
 only on the geometry of the real
semialgebraic set of points satisfying this formula. This construction is achieved in 
Section $2$.     

In order to make some computations more convenient we present a realization, denoted $\chi$,
 of the ring 
$K_0(BSA_{\R})$ in the somewhat more simple ring $K_0(\Var_\R)\otimes \Z[\frac{1}{2}]$, 
that is a
morphism of rings  
$\chi : K_0(BSA_{\R}) \to K_0(\Var_\R)\otimes \Z[\frac{1}{2}],$
that restricts to the identity map on $K_0(\Var_\R)\hookrightarrow K_0(BSA_{\R})$.
The morphism $\chi$ provides an explicit computation (see Proposition \ref{prop-alg}) 
presenting a class of
$K_0(BSA_{\R})$ as a $\Z[\frac{1}{2}]$-linear combination of classes of $K_0(\Var_\R)$.
When one wants to further simplify the computation of a class of a basic real semialgebraic
formula, 
one can shrink the original ring  $K_0(BSA_{\R})$ 
a little bit more  from 
$K_0(\Var_\R)\otimes \Z[\frac{1}{2}]$ to $K_0(\R\Var)\otimes \Z[\frac{1}{2}]$, where
for instance algebraic formulas with empty set of real points have trivial class.
However as noted in point \ref{nontrivial} of Remark \ref{rmks} 
the class of a basic real semialgebraic formula
with empty set of real points may be not trivial in 
$K_0(\R\Var)\otimes \Z[\frac{1}{2}]$. 
The ring $K_0(BSA_{\R})$ is not defined with a prior notion of isomorphism relation 
contrary to the ring $K_0(\Var_\R)$ where algebraic isomorphism classes of varieties
are generators. Nevertheless we indicate a notion of isomorphism for basic semialgebraic 
formulas that factors through $K_0(BSA_{\R})$ (see Proposition \ref{iso}). 
This is done in Section $2$. 

The realization $ \chi : K_0(BSA_{\R}) \to K_0(\Var_\R))\otimes \Z[\frac{1}{2}]$ naturally
allows us to define in Section $4$ a notion of virtual Poincar\'e polynomial for basic real semialgebraic 
formulas: for a class $[F]$ in $K_0(BSA_{\R})$ that is written as a
$\Z[\frac{1}{2}]$-linear combination $ \sum_{i=1}^q a_i[A_i]$ 
of classes $[A_i] \in K_0(\Var_\R)$ of real algebraic varieties $A_i$, 
we simply define 
the virtual Poincar\'e polynomial of $F$ as the corresponding 
$\Z[\frac{1}{2}]$-linear combination $ \sum_{i=1}^q a_i\beta(A_i)$ of virtual Poincar\'e 
polynomials $\beta(A_i)$ of the varieties $A_i$. The virtual Poincar\'e polynomial of 
$F$ is thus a polynomial $\beta(F)$ in $\Z[\frac{1}{2}][u]$. 
It is then shown that the evaluation at $-1$ of $\beta(F)$ is the Euler-Poincar\'e   
 characteristic of the real semialgebraic set of points satisfying the basic formula $F$
(Proposition \ref{eval}). 

These constructions are summed up in the following commutative diagram
\vskip-3mm
 $$ \shorthandoff{;:!?}
\xymatrix{   
Var_\R \ar[d] \ar[rd] \ar@{^{(}->}[rrr]& & &BSA_{\R}  \ar[ddd]^{\chi_c}   \ar[lld] \\
 K_0(\Var_\R) \ar[dd]_\beta \ar@{^{(}->}[r] \hskip2mm\ar@{^{(}->}[rd]   
& K_0(BSA_{\R}) \ar[d]^\chi \ar[rd]^\chi& &\\
& K_0(\Var_\R) \otimes \Z[\frac{1}{2}] \ar[d]^\beta \ar[r] & K_0(\R\Var) \otimes \Z[\frac{1}{2}] 
\ar[ld]^\beta& \\
\Z[u] \ar@{^{(}->}[r] & \Z[\frac{1}{2}] [u] \ar[rr]^{u=-1}& &\Z \\} 
  $$
\vskip0mm
The second and last part of this article concerns the real Milnor fibres of a given 
polynomial function $f\in \R[x_1,\cdots, x_d]$.
As geometrical objects, we consider real semialgebraic Milnor fibres of the following 
types: $f^{-1}(\pm c)\cap \bar B(0,\alpha) $,
$f^{-1}(]0,\pm c[)\cap \bar B(0,\alpha) $,
$f^{-1}(]0,\pm\infty[)\cap S(0,\alpha) $, for 
$0<\vert c \vert \ll\alpha\ll 1$, $ \bar B(0,\alpha)$ the closed ball of 
$\R^d$ of centre $0$ and radius $\alpha$ and $ S(0,\alpha)$ the sphere of centre $0$ 
and radius $\alpha$. The topological types of these fibres are easily comparable, and
in order to present a motivic version of these real semialgebraic Milnor fibres we define appropriate
zeta functions
with coefficients in $(K_0(\Var_\R)\otimes \Z[\frac{1}{2}])[\L^{-1}]$ 
(the localization of the ring $K_0(\Var_\R)\otimes \Z[{1\over 2}]$ with respect to 
the multiplicative set generated by $\L$). As 
in the complex context (see \cite{DL1}, \cite{DL2}), we prove that these zeta functions
are rational functions expressed in terms of an embedded resolution of $f$ (see Theorem 
\ref{Zeta function}). 
For a complex hypersurface $f$, the rationality of the corresponding zeta function 
allows 
the
definition of the motivic Milnor fibre $S_f$, defined as the negative of the limit at 
infinity
of the rational expression of the zeta function. In the real semialgebraic case, the same 
definition makes sense 
but we obtain a class $S_f$ in $K_0(\Var_\R))\otimes \Z[\frac{1}{2}]$ having a realization under
the Euler-Poincar\'e characteristic of greater 
combinatorial complexity in terms of the data of the resolution of $f$ than in the 
complex case. Indeed, all the strata
of the natural stratification of the exceptional divisor of the resolution of $f$ appear in 
the expression of $\chi_c(S_f)$ in the real case. 
Nevertheless we show that the motivic real semialgebraic Milnor fibres have 
for value under the Euler-Poincar\'e characteristic morphism the Euler-Poincar\'e 
characteristic 
of the  corresponding set-theoretic real semialgebraic Milnor fibres 
(Theorem \ref{Milnor}).

In what follows we sometimes simply say {\sl measure} for
the class of an object in a given Grothendieck ring. The term {\sl inequation} refers to the symbol $\not=$,
 and the term {\sl inequality} refers to the symbol $>$.

\tableofcontents
 
%%%%%%%%%%%%%%%%%%%%%%%%%%%%%%%%%%%%%%%%%%%%%%%%%%%%%%%%%%%%%%%%%%%
\section{The Grothendieck ring of basic semialgebraic formulas.}
%%%%%%%%%%%%%%%%%%%%%%%%%%%%%%%%%%%%%%%%%%%%%%%%%%%%%%%%%%%%%%%%%%%

\subsection{Affine real algebraic varieties.}
%%%%%%%%%%%%%%%%%%%%%%%%%%%%%%%%%%%%%%%%%%%%%%%%%%%%%%%%%%%%%

By an affine algebraic variety over $\R$ we mean an affine reduced and
separated scheme of finite type over $\R$. The category of
affine algebraic varieties over $\R$ is denoted by $\Var_\R$.
An affine real algebraic variety $X$ is then defined by a subset of $\A^n$ together
with a finite number of polynomial equations. Namely, there exist $P_i
\in \R[X_1,\ldots,X_n]$, for $i=1,\ldots,r$, such that the real points $X(\R)$
of $X$ are given by
$$X(\R)=\{x\in \A^n | P_i(x)=0,~i=1,\ldots,r\}.$$
A Zariski-constructible subvariety $Z$ of $\A^n$ is similarly defined by real
polynomial equations and inequations. Namely there exist $P_i, Q_j
\in \R[X_1,\ldots,X_n]$, for $i=1,\ldots,p$ and $j=1,\ldots,q$, such
that the real points $Z(\R)$ of $Z$ are given by
$$Z(\R)=\{x\in \A^n | P_i(x)=0, Q_j(x) \neq 0,~i=1,\ldots,p,~j=1,\ldots,q
\}.$$

As an abelian group, the Grothendieck ring $K_0(\Var_{\R})$ of affine real algebraic 
varieties is formally generated by isomorphism classes $[X]$ of
Zariski-constructible real algebraic varieties, 
 subject to the additivity relation
$$[X]=[Y]+[X\setminus Y],$$
in case $Y\subset X$ is a closed subvariety of $X$. Here $X\setminus Y$
is the Zariski-constructible variety defined by combining
the equations and inequations that define $X$ together with the
equations and inequations obtained by 
reversing the equations and inequations that define $Y$.
The product of
constructible sets induces a ring structure on  $K_0(\Var_{\R})$. We
denote by $\L$ the class in $K_0(\Var_{\R})$ of $\A^1$.

%%%%%%%%%%%%%%%%%%%%%%%%%%%%%%%%%%%%%%%%%%%%%%%%%%%%%%%%%%%%%
\subsection{Real algebraic sets.}
%%%%%%%%%%%%%%%%%%%%%%%%%%%%%%%%%%%%%%%%%%%%%%%%%%%%%%%%%%%%%
The real points $X(\R)$ of an affine algebraic variety $X$ over $\R$ form a real
algebraic set (in the sense of \cite{BCR}). The Grothendieck
ring $K_0(\R\Var)$ of affine real algebraic sets
\cite{MCP} is defined in a similar way than that of real algebraic
varieties over $\R$. Taking the real points of an affine real
algebraic variety over $\R$ gives a ring morphism from $K_0(\Var_{\R})$
to $K_0(\R\Var)$. A great advantage of $K_0(\R\Var)$  from a geometrical point of view is that the additivity property implies that the measure of
an algebraic set without real point is zero in $K_0(\R\Var)$.

We already know some realizations of $K_0(\R\Var)$ in simpler rings,
such as the Euler characteristics with compact supports in $\Z$ or the virtual
Poincar\'e polynomial in $\Z[u]$ (cf. \cite{MCP}). We obtain therefore similar
realizations for $K_0(\Var_{\R})$ by composition with the realizations of
$K_0(\Var_{\R})$ in $K_0(\R\Var)$.

%%%%%%%%%%%%%%%%%%%%%%%%%%%%%%%%%%%%%%%%%%%%%%%%%%%%%%%%%%%%%%%%%%%
\subsection{Basic semialgebraic formulas.}
%%%%%%%%%%%%%%%%%%%%%%%%%%%%%%%%%%%%%%%%%%%%%%%%%%%%%%%%%%%%%%%%%%%
Let us now specify the definition of the Grothendieck ring 
$K_0(BSA_{\R})$
of basic semialgebraic formulas. This definition is inspired by \cite{DL3}. 
 The ring $K_0(BSA_{\R})$ will
contain $K_0(\Var_{\R})$ as a subring  (Proposition \ref{incl}) and will be projected 
on the ring $K_0(\Var_{\R})\otimes \Z[\frac{1}{2}]$ 
(Theorem \ref{thm-prin}) by 
an explicit computational process.

A basic semialgebraic formula $A$ in $n$ variables is defined as a 
finite
number of equations, inequations and inequalities, namely there exist
$P_i, Q_j, R_k
\in \R[X_1,\ldots,X_n]$, for $i=1,\ldots,p$, $j=1,\ldots,q$ and
$k=1,\ldots,r$, such that $A(\R)$ is equal to the set of points 
$x\in \A^n$ such that
$$P_i(x)=0, Q_j(x) \neq
0,R_k(x)>0,~i=1,\ldots,p,~j=1,\ldots,q, ~k=1,\ldots,r.$$
The relations $Q_j(x) \neq 0$ are called inequations and the relations 
$R_k(x)>0$ are called inequalities.
We will simply denote a basic semialgebraic formula by 
$$A=\{P_i=0, Q_j\neq 0, R_k>0,~i=1,\ldots,p,~j=1,\ldots,q, ~k=1,\ldots,r \}.$$
In particular   $A$ is not characterized by its  real points $A(\R)$, that is
by the real
solutions of these equations, inequations and inequalities, but by these
equations, inequations and inequalities themselves.  

We will consider basic semialgebraic formulas up to algebraic
isomorphisms, when the basic semialgebraic formulas are defined without
inequality.

\begin{remark} In the sequel, we will allow ourselves to use the notation
  $\{P<0\}$ for the basic semialgebraic formula $\{-P>0\}$ and similarly $\{P>1\}$ instead of
  $\{P-1>0\}$, where $P$ denotes a polynomial with real coefficients. Furthermore given two basic semialgebraic formulas $A$ and $B$, the notation $\{A,B\}$ will denote the basic formula
  with equations, inequations and inequalities coming from $A$
  and $B$ together.
\end{remark}

We define the Grothendieck ring $K_0(BSA_{\R})$ of basic semialgebraic
formulas as the free abelian ring generated by basic semialgebraic formulas $[A]$,
up to algebraic isomorphim when the formula $A$ has no 
inequality, and subject to the three following   relations
\begin{enumerate}
\item (\textit{algebraic additivity}) $$[A]=[A,S=0]+[A,
  \{S\neq 0\}]$$ where $A$ is a basic semialgebraic formula in $n$ variables and
  $S\in \R[X_1,\ldots,X_n]$.
\item (\textit{semialgebraic additivity}) $$[A,R\neq 0]=[A, R>0]+
[A,-R>0]$$ where $A$ is a basic semialgebraic
  formula in $n$ variables and $R\in \R[X_1,\ldots,X_n]$.
  
  \item (\textit{product}) The product of basic semialgebraic formulas, defined by taking the
conjonction of the formulas with disjoint sets of free variables, induces the ring 
product on
$K_0(BSA_{\R})$. In other words we   consider the relation 
$$ [A,B]=[A]\cdot[B], $$
for $A$ and $B$ basic real semialgebraic formulas with disjoint set of variables.
\end{enumerate}

\begin{remark}\label{rmk-iso} 
\begin{enumerate}
\item Contrary to the Grothendieck ring of algebraic varieties or algebraic sets,
  we do not consider isomorphism classes of basic real semialgebraic formulas in
  the definition of $K_0(BSA_{\R})$. As a consequence the realization 
  we are interested in does depend in a crucial way on the
  description of the basic semialgebraic set as a basic semialgebraic formula. 
For instance $\{X-1>0\}$ and $\{X>0,X-1>0\}$ will have different measures.
\item One may decide to enlarge the basic semialgebraic formulas with non-strict inequalities 
by imposing, by convention, that the measure of $\{ A,R \geq 0\}$, for $A$ a basic 
semialgebraic
  formula in $n$ variables and $R\in \R[X_1,\ldots,X_n]$, is the sum of the measures of 
$\{A, R > 0\}$ and of $\{A, R = 0\}$.
\end{enumerate}
\end{remark}

\begin{prop}\label{incl} The natural map $i$ from $K_0(\Var_{\R})$ that associates to 
an affine real algebraic variety its value in the Grothendieck ring $K_0(BSA_{\R})$ 
of basic real semialgebraic formulas is an injective morphism
$$i:K_0(\Var_{\R}) \longrightarrow K_0(BSA_{\R}).$$
\end{prop}

We therefore identify $K_0(\Var_{\R})$ with a subring of $K_0(BSA_{\R}).$

\begin{proof} We construct a left inverse $j$ of $i$ as follows. Let
  $a \in K_0(BSA_{\R})$ be a sum of products of measures of
  basic semialgebraic formulas. If there exist Zariski
  constructible real algebraic sets $Z_1,\ldots,Z_m$ such that
  $[Z_1]+\cdots+[Z_m]$ is equal to $a$ in
  $K_0(BSA_{\R})$, then we define the image of $a$ by
  $j$ to be
$$j(a)=[Z_1]+\cdots+[Z_m] \in K_0(\Var_{\R}).$$
Otherwise, the image of $a$ by
  $j$ is defined to be zero in $K_0(\Var_{\R})$.
%An affine real algebraic variety is naturally a basic
%  semialgebraic set, and for these particular basic semialgebraic
%  sets the notion of isomorphism and additivity coincide with that of
%  $K_0(\Var_{\R})$.
The map $j$ is well-defined. Indeed, if $Y_1,\ldots, Y_l$ are other
Zariski constructible sets such that $[Y_1]+\cdots+[Y_l]$ is equal to
$a$ in $K_0(BSA_{\R})$, then 
$$[Y_1]+\cdots+[Y_l]=[Z_1]+\cdots+[Z_m]$$
in $K_0(BSA_{\R})$. This equality still holds in $K_0(\Var_{\R})$ by
definition of the structure ring of $K_0(\Var_{\R})$ and the fact that $j$ defines a left inverse of $i$ is immediate.
\end{proof}

\begin{remark} Note however that the map $j$ constructed in the proof
  of Proposition \ref{incl} is not a group
morphism. For instance $j([X>0])=j([X<0])=0$ whereas
$j([X\neq 0])=\L-1$.
\end{remark}
%%%%%%%%%%%%%%%%%%%%%%%%%%%%%%%%%%%%%%%%%%%%%%%%%%%%%%%%%%%%%
\section{A realization of $K_0(BSA_{\R})$}
%%%%%%%%%%%%%%%%%%%%%%%%%%%%%%%%%%%%%%%%%%%%%%%%%%%%%%%%%%%%%

An example of a ring morphism from $K_0(BSA_{\R})$ to $\mathbb Z$ is given by
the Euler characteristic with compact supports $\chi_c$. We construct in this section a
realization for elements in $K_0(BSA_{\R})$ with values in the ring of
polynomials with coefficient in $\Z[\frac{1}{2}]$. This realization specializes to the 
Euler characteristic with compact supports. To
this aim, we construct a ring morphism from
$K_0(BSA_{\R})$ to the tensor product of $K_0(\Var_{\R})$ with
$\Z[\frac{1}{2}]$.

%%%%%%%%%%%%%%%%%%%%%%%%%%%%%%%%%%%%%%%%%%%%%%%%%%%%%%%%%%%%%%%%%%%
\subsection{The realization.}
%%%%%%%%%%%%%%%%%%%%%%%%%%%%%%%%%%%%%%%%%%%%%%%%%%%%%%%%%%%%%%%%%%%

We define a morphism $\chi$ from the ring $K_0(BSA_{\R})$ to the ring $K_0(\Var_{\R})\otimes \Z[\frac{1}{2}]$ as follows. Let
$A$ be a basic semialgebraic formula without inequality. We assign to $A$ its value $\chi(A)=[A]$ in $K_0(\Var_{\R})$ as a 
constructible set. We proceed now by induction on the number of
inequalities in the description of the basic semialgebraic formulas. 
Assuming that we have defined $\chi$ for basic
semialgebraic formulas with at most $k$ inequalities, $k\in \N$, let $A$
be  a basic real semialgebraic formula with $n$ variables and at most $k$ inequalities and let us consider $R\in
\R[X_1,\ldots,X_n]$. Define $\chi([A, R>0])$ by
$$\chi([A, R>0]):=\frac{1}{4}\big(\chi([A, Y^2=R])-\chi([A,
Y^2=-R])\big)+\frac{1}{2} \chi([A, R\neq 0]),$$
where $\{A, Y^2=\pm R\}$ is a basic real semialgebraic formula with $n+1$ variables, 
with at most $k$ inequalities and $\{A, R\neq 0\}$ is a 
basic semialgebraic formula with $n$ variables with at most $k$ inequalities.

\begin{remark}\label{rmk-ori}
The way to define $\chi$ may be seen as an average of two different
natural ways of understanding a basic semialgebraic formula as a quotient
of algebraic varieties. Namely, for a basic semialgebraic formula in
$n$ variables of the form $\{R>0\}$, we may see its set of real 
points as the projection, with fibre two points,  of
$\{Y^2=R\}$ minus the zero set of $R$, or as the complement 
of the projection of $Y^2=-R$. The algebraic
average of these two possible points of view is 
$$\frac{1}{2}\Big(\big(\frac{1}{2}[Y^2=R]-[R=0]\big)+ \big( \L^n-\frac{1}{2}[Y^2=-R]\big) \Big),$$
which, considering that $ \L^n-[R=0]= [R\not=0]$, gives for $\chi(R>0)$ the expression 
just defined above. 
\end{remark}

We give below the general  formula that computes the measure of a basic
semialgebraic formula in terms of the measure of real algebraic varieties.

\begin{prop}\label{prop-alg} Let $Z$ be a constructible set in $\R^n$ and take $R_k\in
  \R[X_1,\ldots,X_n]$, with $k=1,\ldots,r$. For $I\subset
  \{1,\ldots,r\}$ a subset of cardinal $\sharp I=i$ and $\eps \in \{\pm
  1\}^i$, we denote by $R_{I,\eps}$ the real constructible set defined by
$$R_{I,\eps}=\{Y_j^2=\eps_jR_j(X),j\in I;~~R_k(X)\neq 0, k\notin I\}.$$
Then
  $\chi([Z,R_k>0,~k=1,\ldots,r])$ is equal to
$$\sum_{i=0}^r \frac{1}{2^{r+i}}\sum_{I\subset
  \{1,\ldots,r\},\sharp I=i}\sum_{\eps \in \{\pm 1\}^i}(\prod_{j\in I}\eps_j)
[Z,R_{I,\eps}]$$
\end{prop}

\begin{proof} If $r=1$ it follows from the
  definition of $\chi$. We prove the general result by induction on
  $r\in \N$. Assume $Z=\R^n$ to simplify notation. Take  $R_k\in \R[X_1,\ldots,X_n]$, with
  $k=1,\ldots,r+1$. 
Denote by $A$ the formula $R_1>0,\ldots, R_r>0$. By definition of $\chi$ we obtain
$$\chi([A,R_{r+1}>0])=\frac{1}{4}(\chi([A, Y^2=R_{r+1}])-\chi([A, Y^2=-R_{r+1}]))+\frac{1}{2}\chi([A, R_{r+1}\neq 0]).$$
Now we can use the induction assumption to express the terms in the right hand side of the formula upstair as
$$\sum_{i=0}^r \frac{1}{2^{r+i}}\sum_{I\subset
  \{1,\ldots,r\},\sharp I=i}\sum_{\eps \in \{\pm 1\}^i}(\prod_{j\in I}\eps_j) \big(\frac{1}{4}([R_{I,\eps}, Y^2=R_{r+1}]-[R_{I,\eps}, Y^2=-R_{r+1}])$$
  $$+\frac{1}{2}[R_{I,\eps}, R_{r+1}\neq 0]  \big) $$
Choose $I\subset
  \{1,\ldots,r\}$ a subset of cardinal $\sharp I=i$ and $\eps \in \{\pm
  1\}^i$. Then, we obtain from the definition of $\chi$ that
$$\frac{1}{4}([R_{I,\eps}, Y^2=R_{r+1}]-[R_{I,\eps}, Y^2=-R_{r+1}])+\frac{1}{2}[R_{I,\eps}, R_{r+1}\neq 0]$$
is equal to 
$$\frac{1}{4}([R_{I\cup \{r+1\},\eps^+}]-[R_{I\cup \{r+1\},\eps^-}])+\frac{1}{2}[R_{\tilde I,\eps}]$$
where $\eps^+=(\eps_1,\ldots,\eps_r,1)$,
$\eps^-=(\eps_1,\ldots,\eps_r,-1)$ and $\tilde I$ denotes $I$ as a
subset of $\{1,\ldots,r+1\}$. Therefore 
$$\frac{1}{2^{r+i}}(\prod_{j\in I}\eps_j)[R_{r+1}>0,
R_{I,\eps}]$$
is equal to
$$\frac{1}{2^{(r+1)+(i+1)}}(\prod_{j\in I}\eps_j)([R_{I\cup
  \{r+1\},\eps^+}]-[R_{I\cup
  \{r+1\},\eps^-}])+\frac{1}{2^{(r+1)+i}}(\prod_{j\in
    I}\eps_j)[R_{\tilde I,\eps}]$$
which gives the result. 
\end{proof}

The morphism $\chi$ is then  defined on $K_0(BSA_{\R})$.

\begin{thm}\label{thm-prin} The map $$\chi: K_0(BSA_{\R}) \longrightarrow
  K_0(\Var_{\R})\otimes \Z[\frac{1}{2}]$$ is a ring morphism that is
  the identity on $K_0(\Var_{\R})\subset K_0(BSA_{\R})$.
\end{thm}

\begin{proof} We must prove that the given definition of $\chi$ is
  compatible with the algebraic and semialgebraic additivities. However the semialgebraic additivity follows directly from the definition of
$\chi$. Indeed, if $A$ is a basic semialgebraic formula and $R$ a real
polynomial, then the sum of $\chi([A ,R>0])$ and $\chi([A ,-R>0])$ is equal to
$$\frac{1}{4}\big(\chi([A,Y^2=R])-\chi([A,Y^2=-R])\big)+\frac{1}{2}\chi([A,R\neq 0])$$
$$+\frac{1}{4}\big(\chi([A,Y^2=-R])-\chi([A,Y^2=R])\big)+\frac{1}{2}
\chi([A,-R\neq 0])$$
$$=\chi([A,-R\neq 0]).$$

The algebraic additivity as well as the multiplicativity follow from Proposition \ref{prop-alg} that
enables to express the measure of a basic semialgebraic formula in terms
of algebraic varieties for which additivity and multiplicativity hold.
We conclude by noting that we may construct a left inverse to $\chi$
restricted to $K_0(\Var_{\R})$ in the same way as in the proof of
Proposition \ref{incl}.
\end{proof}

\begin{ex}\label{ex1}
\begin{enumerate}
\item A half-line defined by $X>0$ has measure in $K_0(\Var_{\R})\otimes
\Z[\frac{1}{2}]$ half of the value of the line minus one point, as expected, since
by definition 
$$\chi([X>0])=\frac{1}{4}(\L-\L)+\frac{1}{2}\big(\L-1)=\frac{1}{2}\big(\L-1).$$
However, if we add one more inequality, like $\{X>0,X>-1\}$, then
the measure has more complexity. We will see in section \ref{sect-virt} that, evaluated in the polynomial ring $\Z[\frac{1}{2}][u]$ we
obtain in that case
$$\beta([X>0,X>-1])=\frac{5u-11}{16}.$$
\item Using the multiplicativity, we find the measure of the
  half-plane and the measure of the quarter
  plane as expected
$$\chi([X_1>0])=\frac{1}{2}(\L^2-\L)$$
and
$$\chi([X_1>0,X_2>0])=\frac{1}{4}(\L-1)^2.$$
\end{enumerate}
\end{ex}

\begin{remark}\label{rmks}
\begin{enumerate}
\item  Let $R\in \R[X_1,\ldots,X_n]$ be odd. Then 
$$\chi([R>0])=\chi([R<0])=\frac{[R\neq 0]}{2}.$$ 
Indeed, the varieties $Y^2=R(X)$ and $Y^2=-R(X)$
  are isomorphic via $X\mapsto -X$, and the result follows from the
  definition of $\chi$.
\item\label{nontrivial} The ring morphism from $K_0(\Var_{\R})$ to
  $K_0(\R\Var)$ gives a realization from the ring $K_0(BSA_{\R})$ to the ring 
$K_0(\R\Var)\otimes \Z[\frac{1}{2}]$
for which the measure of a
  real algebraic variety without real point is zero, this is why it is often 
convenient to push the computations to the ring  
$K_0(\R\Var)\otimes \Z[\frac{1}{2}]$ rather than staying at the higher level
of $K_0(\Var_{\R})\otimes \Z[\frac{1}{2}]$. However we have to notice that the
  measure of a basic real semialgebraic formula without real point is not
  necessarily zero in $K_0(\R\Var)\otimes \Z[\frac{1}{2}]$. 
For instance, let us compute the measure of $X^2+1>0$ in 
$K_0(\R\Var)\otimes \Z[\frac{1}{2}]$. 
By  definition of $\chi$ we obtain that $\chi([X^2+1>0])$ is equal to
$$\frac{1}{4}\big(\chi([Y^2=X^2+1])-\chi([Y^2=-X^2-1])\big)+\frac{1}{2}\chi([X^2+1
\neq 0])$$
$$=\frac{1}{4}(\L-1)+\frac{1}{2}\L=\frac{1}{4}(3\L-1).$$
By additivity we have 
$$ \chi([X^2+1<0])=\chi([X^2+1\not=0])-\chi([X^2+1>0])$$
$$ =\L-\chi([X^2+1=0])-\chi([X^2+1>0]). $$
But since $\chi([X^2+1=0])=0$ in $K_0(\R\Var)\otimes \Z[\frac{1}{2}]$,
we obtain that the measure of $\{X^2+1<0\}$ in $K_0(\R\Var)\otimes \Z[\frac{1}{2}]$, whose real points set is empty, is 
$$\chi([X^2+1<0])=\frac{1}{4}(\L+1).$$

\item In a similar way, the basic semialgebraic formula $\{P>0,-P>0\}$ with
  $P(X)=1+X^2$, whose set of real points is empty, has measure
$$\chi([P>0,-P>0])=\frac{1}{8}(\L+1).$$
\end{enumerate}
\end{remark}

%\begin{ex} The result is no longer true if $r\geq 2$. For example 
%$$\chi(\{X_1+X_2=1, X_1>0,X_2>0\})=\frac{1}{16}(3\L+3).$$
%\end{ex}

%%%%%%%%%%%%%%%%%%%%%%%%%%%%%%%%%%%%%%%%%%%%%%%%%%%%%%%%%%%%%%%%%%%%%%%%
\subsection{Isomorphism between basic semialgebraic formulas}
%%%%%%%%%%%%%%%%%%%%%%%%%%%%%%%%%%%%%%%%%%%%%%%%%%%%%%%%%%%%%%%%%%%%%%%%
In this section we give a condition for two basic semialgebraic formulas to have the same realization by $\chi$. It deals with the
complexification of the algebraic liftings of the basic semialgebraic
formulas. 

Let $X$ be a
real algebraic subvariety of $\mathbb R^n$ defined by $P_i
\in \R[X_1,\ldots,X_n]$, for $i=1,\ldots,r$. The complexification
$X_{\C}$ of $X$ is defined to be the complex algebraic subvariety of
$\C^n$ defined by the same polynomials $P_1,\ldots,P_r$. We define
similarly the complexification of a real algebraic map.

Let $Y\subset \R^n$ be a Zariski
  constructible subset of $\R^n$ and take $R_1,\ldots,R_r
\in \R[X_1,\ldots,X_n]$. Let $A$ denote the basic
semialgebraic formula of $\R^n$ defined by $Y$ together with the inequalities
$R_1>0,\ldots,R_r>0$, and $V$ denote the Zariski constructible subset
of $\R^{n+r}$ defined by 
$$V=\{Y,Y_1^2=R_1,\ldots,Y_r^2=R_r\}.$$
Note that $V$ is endowed with an action
of $\{\pm 1\}^r$ defined by multiplication by $-1$ on the
indeterminates $Y_1,\ldots,Y_r$.

 Let $Z\subset \R^n$ be a Zariski constructible subset of $\R^n$ and
 take similarly $S_1,\ldots,S_r
\in \R[X_1,\ldots,X_n]$. Let $B$ denote the basic
semialgebraic formula of $\R^n$ defined by $Z$ together with the inequalities
$S_1>0,\ldots,S_r>0$, and $W$ denote the Zariski constructible subset
of $\R^{n+r}$ defined by 
$$W=\{Z,Y_1^2=S_1,\ldots,Y_r^2=S_r\}.$$

\begin{definition} 
We say that the basic semialgebraic formulas $A$ and $B$ are isomorphic
if there exists a real algebraic isomorphism $\phi:V \longrightarrow
W$ between $V$ and $W$ which is equivariant with respect to the action
of $\{\pm 1\}^r$ on $V$ and $W$, and whose complexification
$\phi_{\C}$ induces a complex algebraic isomorphism between the
complexifications $V_{\C}$ and $W_{\C}$ of $V$ and $W$.
\end{definition}

\begin{remark}\label{rmk-1} Let us consider first the particular case  $Y=\R^n$, $Z=\R^n$ and $r=1$. 
Change moreover the notation as
follows. Put $V^+=V$ and $W^+=W$, and define $V^-=\{y^2=-R(x)\}$ and
$W^-=\{y^2=-S(x)\}$. 

Then the complex points $V_{\mathbb C}^+$ and $V_{\mathbb C}^-$ of
$V^+$ and $V^-$ are isomorphic via the complex (and not real)
isomorphism $(x,y)\mapsto (x,iy)$. Now, suppose that the basic
semialgebraic formula $\{R>0\}$ is isomorphic to $\{S>0\}$. Let $\phi=(f,g):(x,y)\mapsto (f(x,y),g(x,y))$ be the real isomorphism involved in the definition (that is $f$ and $g$ are defined by real equations, 
and moreover $f(x,-y)=f(x,y)$ and $g(x,-y)=-g(x,y)$). Then the following diagram 

$$ \begin{matrix} 
V^+_{\mathbb C}   &\buildrel{(f,g)}\over \longrightarrow & W^+_{\mathbb C}\\
 \buildrel{(x,y)\mapsto (x,iy)}\over \downarrow  &   &  \buildrel{(x,y)\mapsto (x,iy)}\over \downarrow \\
V^-_{\mathbb C}   & & W^-_{\mathbb C}\\ 
   \end{matrix} $$
induces a complex isomorphism $(F,G)$ between $ V^-_{\mathbb C}$ and $W^-_{\mathbb C}$ given by
$$(x,y) \mapsto (f(x,-iy),ig(x,-iy)).$$
In fact, this isomorphism is defined over $\mathbb R$ since
$$\overline {F(x,y)}=\overline {f(x,-iy)}=f(\overline x, \overline {-iy})=f(\overline x,i \overline y)=f(\overline x,-i \overline y)=F(\overline x,\overline y)$$ and
$$\overline {G(x,y)}=\overline {ig(x,-iy)}=-ig(\overline x, \overline
{-iy})=-ig(\overline x,i \overline y)=ig(\overline x,-i \overline
y)=G(\overline x,\overline y),$$
where the bar denotes complex conjugation. Therefore it induces a real
algebraic isomorphism between $V^-$ and $W^-$.

Moreover $g(x,0)=-g(x,0)$ so $g(x,0)=0$ and then the real algebraic sets $\{R=0\}$ and $\{S=0\}$ are also isomorphic.
\end{remark}

\begin{prop}\label{iso} If the basic semialgebraic formulas $A$ and $B$ are
  isomorphic, then $\chi([A])=\chi([B])$.
\end{prop}

\begin{proof} Thanks to Proposition \ref{prop-alg}, we only need to prove that
  the real algebraic varieties $R_{I,\eps}$ corresponding to $A$ and
  $B$ are isomorphic two by two, which is a direct generalization of
  Remark \ref{rmk-1}.
\end{proof}

%%%%%%%%%%%%%%%%%%%%%%%%%%%%%%%%%%%%%%%%%%%%%%%%%%%%%%%%%%%%%%%%%%%
\section{Virtual Poincar\'e polynomial}
%%%%%%%%%%%%%%%%%%%%%%%%%%%%%%%%%%%%%%%%%%%%%%%%%%%%%%%%%%%%%%%%%%%

\subsection{Polynomial realization}\label{sect-virt}

The best realization known (with respect to the highest 
algebraic complexity of the realization ring) of the Grothendieck ring of real algebraic varieties is
given by the virtual Poincar\'e polynomial \cite{MCP}. This
polynomial, whose coefficients coincide with the Betti numbers with
coefficients in $\frac{\Z}{2\Z}$ when sets are compact and nonsingular,
has coefficient in $\Z$. As a corollary of Theorem \ref{thm-prin} we
obtain the following realization of $K_0(BSA_{\R})$ in $\Z[\frac{1}{2}][u]$.

\begin{prop} There exists a ring morphism
$$\beta:K_0(BSA_{\R}) \longrightarrow \Z[\frac{1}{2}][u]$$
whose restriction to $K_0(\Var_{\R})\subset K_0(BSA_{\R})$ coincides with
the virtual Poincar\'e polynomial.
\end{prop}

The interest of such a realization is that it enables to make concrete computations.

\begin{ex}
\begin{enumerate}
\item The virtual Poincar\'e polynomial of the open disc $X_1^2+X_2^2<1$ is equal to 
$$\frac{1}{4}\big(\beta([Y^2=1-(X_1^2+X_2^2)])-\beta([Y^2=X_1^2+X_2^2-1])\big)+
\frac{1}{2}\beta([X_1^2+X_2^2\neq 1])$$
$$=\frac{1}{4}(u^2+1-u(u+1))+\frac{1}{2}(u^2-u-1)=\frac{1}{4}(2u^2-3u-1).$$
\item Let us compute the measure of the formula $X>a,X>b$ with $a\neq b \in \mathbb R$. By Proposition \ref{prop-alg}, we are lead to compute the virtual Poincar\'e polynomial of the real algebraic subsets of $\mathbb R^3$ defined by $\{y^2=\pm (x - a),~~z^2=\pm (x - b)\}$. These sets are isomorphic to $\{y^2 \pm z^2 = \pm (a - b)\}$, and we recognise either a circle, a hyperbola or the emptyset.

In particular, using the formula in Proposition \ref{prop-alg}, we obtain
$$\beta([X>a,X>b])=\frac{1}{16}(2(u-1)-(u+1))+\frac{1}{8}(2u-2u)+\frac{1}{8}(2-2)+\frac{1}{4}(u-2)=\frac{5u-11}{16}$$

\end{enumerate}
\end{ex}

\begin{remark} In case the set of real points of a basic semialgebraic
  formula is a real algebraic set (or even an arc symmetric set \cite{KK,F}), its
  virtual Poincar\'e polynomial does not coincide in general with the
  virtual Poincar\'e polynomial of the real algebraic set. For
  instance, the basic semialgebraic formula $X^2+1>0$, considered in
  Remark \ref{rmks}, has virtual Poincar\'e
  polynomial equal to
  $\frac{1}{4}(3u-1)$ whereas its set of points is a real line whose
  with virtual Poincar\'e polynomial equals  $u$ as a real algebraic set.
\end{remark}

Evaluating $u$ at an integer gives another realization, with
coefficient in $\Z[\frac{1}{2}]$.
The virtual Poincar\'e polynomial of a real algebraic variety, evaluated
at $u=-1$, coincides with its Euler characteristic with compact
supports \cite{MCP}. Indeed, evaluating the virtual Poincar\'e
polynomial of a basic semialgebraic formula gives also the Euler characteristic with compact
supports of its set of real points, and therefore has its values in $\Z$.

\begin{prop}\label{eval} The virtual Poincar\'e
polynomial $\beta(A)$ of a basic semialgebraic formula $A$ is equal to
the Euler characteristic with compact supports of its set of real
points $A(\R)$
when evaluated at $u=-1$. In other words
$$\beta(A)(-1)=\chi_c(A(\R)).$$
\end{prop}

\begin{proof} We recall that in Proposition \ref{prop-alg} we explain how to express the class of $A$ as a 
linear combination 
of classes of real algebraic varieties for which the virtual Poincar\'e polynomial evaluated
at $u=-1$ coincides with the Euler characteristic with compact
supports. At each step of our inductive process to obtain such a linear combination, we introduce a new
variable and a double covering of the set of points satisfying one less inequality. The inductive formula

$$\chi([B, R>0]):=\frac{1}{4}\big(\chi([B, Y^2=R])-\chi([B,
Y^2=-R])\big)+\frac{1}{2} \chi([B, R\neq 0]),$$
used at this step to eliminate one inequality by replacing the system $\{B,R>0\}$ by other systems $\{B, Y^2=R\}, \{B, Y^2=-R\}, 
\{B, R\neq 0\}$ is compatible with the Euler characteristic of the underlying sets of points, that is to say that 
our induction formula is true for $\chi=\chi_c$. The geometric reason for this fact is explained in Remark \ref{rmk-ori}, and is 
the intuitive motivation to define the realization $\chi$ by induction precisely as it is defined.

%The results follows from the fact that for these varieties, constructed as double coverings, 
%the Euler characteristic with compact supports of the total space is twice that of the basis.
%As explained in Remark \ref{rmk-ori}, one may see
%  $\chi(A)$ as an average of two natural $2$-coverings of the set of
%  real points $A(\R)$ of $A$. However, for such coverings, the Euler
%  characteristic with compact supports of the total space is twice
%  that of the basis.
\end{proof}

\subsection{Homogeneous case}
We propose some computations of the virtual Poin\-ca\-r\'e polynomial of
basic real semialgebraic formulas of the form $\{R>0\}$ where $R$ is
homogeneous. Looking at Euler characteristic with compact supports, it
is equal to the product of the Euler characteristics with compact
supports of $\{X>0\}$ with
$\{R=1\}$. We investigate the case of virtual Poincar\'e polynomial. A
key point in the proofs will be the invariance of the virtual Poincar\'e polynomial
of constructible sets under regular homeomorphisms (see \cite{MCP2}, Proposition 4.3).

\begin{prop}\label{homo-odd} Let $R\in \R[X_1,\ldots,X_n]$ be a homogeneous polynomial
  of degre $d$. Assume $d$ is odd. Then
$$\beta([R>0])=\beta([X>0])\beta([R=1]).$$
\end{prop}

\begin{proof} The algebraic varieties defined by $Y^2=R(X)$ and
  $Y^2=-R(X)$ are isomorphic since $R(-X)=-R(X)$, therefore
  $$\beta([R>0])=\frac{\beta([R\neq 0])}{2}.$$
The map $(\lambda,x)\mapsto \lambda x$ from $\mathbb R^* \times
\{R=1\}$ to $R\neq 0$ is a regular homeomorphism with inverse $y
\mapsto (R(y)^{1/d}, \frac{y}{R(y)^{1/d}})$ therefore 
$$\beta([R\neq 0])=\beta(\mathbb R^*)\beta([R=1]),$$
so that 
$$\beta([R>0])=\frac{\beta(\mathbb R^*)}{2}\beta([R=1])=\beta([X>0])\beta([R=1]).$$
\end{proof}

The result is no longer true when the degre is even. However, in the
particular case of the square of a homogeneous polynomial of odd
degre, the relation of Proposition \ref{homo-odd} remains valid.

\begin{prop} Let $P\in \R[X_1,\ldots,X_n]$ be a homogeneous polynomial
  of degre $k$. Assume $k$ is odd, and define $R\in
  \R[X_1,\ldots,X_n]$ by $R=P^2$. Then
$$\beta([R>0])=\beta([X>0])\beta([R=1]).$$
\end{prop}

\begin{proof} Note first that $\{Y^2-R\}$ can be factorized as
  $(Y-P)(Y+P)$ therefore the virtual Poincar\'e polynomial of
  $Y^2-R$ is equal to
$$\beta(Y-P=0 )+\beta(Y+P=0)-\beta(P=0).$$
However the algebraic varieties $Y-P=0$ and $Y+P=0$ are
isomorphic to a $n$-dimensional affine space, whereas $Y^2+R=0$
is isomorphic to $P=0$ since $R=P^2$ is positive, so
that the virtual Poincar\'e polynomial of $R>0$ is equal to
$$\frac{1}{4}(2\beta(\mathbb
R^n)-2\beta([P=0]))+\frac{1}{2}\beta([P\neq 0])=\beta([P\neq 0]).$$

To compute $\beta([P\neq 0]$, note that the map 
$(\lambda,x)\mapsto \lambda x$ from 
$\mathbb R^* \times \{P=1\}$ to $\{P\neq 0\}$ is a regular homeomorphism with inverse $y
\mapsto (R(y)^{1/k}, \frac{y}{R(y)^{1/k}})$ therefore 
$$\beta([P\neq 0])=\beta(\mathbb R^*)\beta([P=1]).$$
We achieve the proof by noticing that $R-1=(P-1)(P+1)$ so that
$\beta([P=1])=\frac{\beta([R=1])}{2}$ because the degree of the
homogeneous polynomial $P$ is odd.
Finally
 $$\beta([R>0])=\frac{\beta(\mathbb R^*)}{2}\beta([R=1])$$
and the proof is achieved.
\end{proof}

More generally, for a homogeneous polynomial $R$ of degre twice a odd number, we can 
express the
virtual Poincar\'e polynomial of $[R>0]$ in terms of that of
$[R=1]$, $[R=-1]$ and $[R\neq 0]$ as follows.

\begin{prop}\label{homo-even} Let $k\in \mathbb N$ be odd and put $d=2k$. 
Let $R\in \R[X_1,\ldots,X_n]$ be a homogeneous polynomial
  of degre $d$. Then
$$\beta([R>0])=\frac{1}{4}\beta(\mathbb R^*)(\beta([R=1])-\beta([R=-1]))+\frac{1}{2}\beta([R\neq 0]).$$
\end{prop}

\begin{ex} We cannot do better in general as illustrated by the
  following examples. For $R_1=X_1^2+X_2^2$ one obtain
$$\beta([R_1>0])=\frac{3}{2}\beta([X>0])\beta([R_1=1])$$
whereas for $R_2=X_1^2-X_2^2$ one has
$$\beta([R_2>0])=\beta([X>0])\beta([R_2=1]).$$
\end{ex}

The proof of Proposition \ref{homo-even} is a direct consequence of the next lemma.

\begin{lemma}  Let $k\in \mathbb N$ be odd and put $d=2k$. Let $R\in \R[X_1,\ldots,X_n]$ be a homogeneous polynomial
  of degre $d$. Then
$$\beta([Y^2=R])=\beta([R= 0])+\beta(\mathbb R^*)\beta([R=1]).$$
\end{lemma}

\begin{proof} Note first that the algebraic varieties $Y^2=R$ and
  $Y^d=R$ have the same virtual Poincar\'e polynomial. Indeed the
  map $(x,y)\mapsto (x,y^k)$ realizes a regular homeomorphism between
  $Y^2=R$ and
  $Y^d=R$, whose inverse is given by $(x,y)\mapsto (x,y^{1/k})$. 
However the polynomial $Y^d-R$ being homogeneous, we obtain a regular homeomorphism
$$\mathbb R^* \times (\{R=1\}\cap \{Y^d=R\}) \longrightarrow \{R\neq 0\}\cap \{Y^d=R\}$$
defined by $(\lambda,x,y) \mapsto (\lambda x,\lambda y)$. As a consequence
$$\beta([Y^d-R=0])=\beta([R= 0])+\beta(\mathbb R^*)\beta([R=1]).$$
\end{proof}

%%%%%%%%%%%%%%%%%%%%%%%%%%%%%%%%%%%%%%%%%%%%%%%%%%%%%%%%%%%%%
\section{Zeta functions and Motivic real Milnor fibres}
%%%%%%%%%%%%%%%%%%%%%%%%%%%%%%%%%%%%%%%%%%%%%%%%%%%%%%%%%%%%%
 
We apply in this section the preceding construction of 
$\chi : K_0(BSA_{\R})\to K_0(\Var_\R)\otimes 
\Z[\frac{1}{2}]$ in defining, for a given polynomial $f\in \R[X_1,\cdots, X_d]$, 
zeta functions whose coefficients are classes 
 in
$(K_0(\Var_\R)\otimes \Z[\frac{1}{2}])[\L^{-1}]$
of real semialgebraic formulas in truncated arc spaces. 
We then show that these
zeta functions are deeply related to the topology of some 
corresponding set-theoretic  real semialgebraic Milnor fibres of~$f$.
%%%%%%%%%%%%%%%%%%%%%%%%%%%%%%%%%%%%%%%%%%%%%%%%%%%%%%%%%%%%%%%%%%%
\subsection{Semialgebraic zeta functions and real Denef-Loeser formulas.}
%%%%%%%%%%%%%%%%%%%%%%%%%%%%%%%%%%%%%%%%%%%%%%%%%%%%%%%%%%%%%%%%%%%
Let $f:\R^d \to \R$ be a polynomial function with coefficients in $\R$ sending $0$
to $0$.
We denote by  ${\cal L}$ or $\cal L(\R^d,0)$ the space of formal arcs
$\gamma(t)=(\gamma_1(t), \cdots, \gamma_d(t))$ on $\R^d$, with $\gamma_j(0)=0$ for all 
$j\in \{1, \cdots, d\}$, by $\cal L_n$ or $\cal L_n(\R^d,0)$ the space of 
truncated arcs ${\cal L}/(t^{n+1})$ and by $\pi_n : \cal L \to \cal L_n$
the truncation map. More generally, for $M$ a variety and $W$ a closed subset of $M$, 
$\cal L(M,W)$ (resp. $\cal L_n(M,W)$) will denote the 
space of arcs on $M$ (resp. the $n$-th jet-space on $M$) 
with endpoints in $W$. 

Let $\epsilon$ be one of the symbols in the set 
$\{ \hbox{\sl naive}, -1, 1, >, < \}$. 
For such a symbol $\epsilon$, via the 
realization of $K_0(BSA_{\R})$ in 
$K_0(\Var_\R)\otimes \Z[\frac{1}{2}]$, we 
define a zeta function $Z_f^\epsilon(T)\in (K_0(\Var_\R)\otimes 
\Z[\frac{1}{2}])[\L^{-1}][[T]]$ by
$$ Z_f^\epsilon(T):=\sum_{n\ge 1}\ [X_{n,f}^\epsilon]\L^{-nd}T^n,$$
 where $X_{n,f}^\epsilon$ is defined in the following way:
 \vskip0,3cm 
 - $X_{n,f}^{naive}=\{\gamma \in {\cal L}_n; 
\ f(\gamma(t))=at^n+\cdots, a\not=0\}$, 
  \vskip0,1cm 
 - $X_{n,f}^{-1}=\{\gamma \in {\cal L}_n; 
\ f(\gamma(t))=at^n+\cdots, a=-1\}$,
   \vskip0,1cm
 - $X_{n,f}^{1}=\{\gamma \in {\cal L}_n; 
\ f(\gamma(t))=at^n+\cdots, a=1\}$,
   \vskip0,1cm
 - $X_{n,f}^>=\{\gamma \in {\cal L}_n; 
\ f(\gamma(t))=at^n+\cdots, a >0\}$,
  \vskip0,1cm
- $X_{n,f}^<=\{\gamma \in {\cal L}_n; 
\ f(\gamma(t))=at^n+\cdots, a<0\}$.
\vskip3mm
Note that $X_{n,f}^\epsilon$ is a real algebraic variety for $\epsilon = -1$ or $1$, 
a real algebraic constructible set for $\epsilon =naive$ 
and a semialgebraic set, given by an explicit description involving one inequality, for $\epsilon$ being the symbol $ >$ or the symbol $<$. 
Consequently, 
$Z_f^\epsilon (T)\in  K_0(\Var_\R)[\L^{-1}][[T]]$ for  $\epsilon \in \{naive, -1, 1\}$  and 
 $Z_f^\epsilon (T)\in (K_0(\Var_\R)\otimes 
\Z[\frac{1}{2}])[\L^{-1}][[T]]$ for  $\epsilon \in \{ > , <\}$.

We show in this section that $Z_f^\epsilon(T)$ is a rational function expressed in terms 
of the combinatorial data of a resolution of $f$. To define those data let us consider 
$\sigma : (M,\sigma^{-1}(0))\to 
(\R^d,0)$  a proper birational map which is an isomorphism over the complement
 of $\{f=0\}$ in $(\R^d,0)$, such that $f\circ \sigma$ and  the jacobian determinant
$\jac\ \sigma$ are normal crossings and $\sigma^{-1}(0)$ is a union of components of the exceptional divisor. We denote by 
$ E_j$, for $ j\in \cal J$, the irreducible components of 
$(f\circ \sigma)^{-1}(0)$ and
assume that $E_k$ are the irreducible components of 
$\sigma^{-1}(0)$ for $k\in \cal K \subset \cal J$. For $j\in \cal J$ 
we denote by $N_j$ the multiplicity 
 $mult_{E_j}f\circ \sigma $ of 
$f\circ \sigma$ along $E_j$ and for $k\in \cal K$ by $\nu_k$ the number $\nu_k=
1+mult_{E_k}\jac \ \sigma$. For any $I\subset \cal J$, we put 
$E^0_I=(\bigcap_{i\in I} E_i)\setminus (\bigcup_{j\in \cal J\setminus I}E_j)$.
These sets $E^0_I$ are constructible sets and the collection $(E_I^0)_{I\subset \cal J}$ gives a canonical stratification of the divisor
$f\circ \sigma=0$, compatible with $\sigma=0$ such that in some affine open subvariety $U$ 
 in $M$ we have $f\circ \sigma (x)=u(x) \prod_{i\in I} x_i^{N_i}$, where
$u$ is a unit, that is to say a rational function which does not vanish on $U$, and
$x=(x',(x_i)_{i\in I})$ are local coordinates.

Finally for $\epsilon \in \{-1, 1, >,<\}$ and $I\subset \cal J$, 
we define $\widetilde E^{0,\epsilon}_I$ as the gluing along $E^0_I$ of the sets  
$$ R_U^\epsilon=\{ (x,t)\in (E^0_I\cap U)\times \R; \ t^m \cdot u(x)\ ?_\epsilon \ \},$$
where $?_\epsilon$ is $=-1$, $=1$, $>0$ or $<0$ in case $\epsilon$ is $-1,1, >$ or 
$< $ and $m=gcd_{i\in I}(N_i)$.

\begin{remark}\label{gluing} The definition of the $R_U^\epsilon$'s is 
independent of the choice of the coordinates, as well as 
the gluing of the $R^\epsilon_U$ is allowed, up to isomorphism, since when in some Zariski neighborhood of  
$E^0_I$ one has in another coordinate system $z=z(x)=(z', (z_i)_{i\in I})$ the expression $f\circ \sigma (z)= v(z)\prod_{i\in I}
z^{N_i}$,  there exist non-vanishing functions $\alpha_i$
so that $z_i=\alpha_i(z)\cdot x_i$. We thus obtain  
$v(z)\prod_{i\in I}\alpha_i^{N_i}(z)=u(x)$, and the transformation

$$ \begin{matrix}  
     & \{(x,t)\in (E_I^0\cap U)\times \R; t^m \cdot u(x)\ ?_{\epsilon}  \} & \to & 
\{(z,s)\in (E_I^0\cap U)\times \R; s^m\cdot v(z) \ ?_{\epsilon}  \} \\
   & \hfill (x,t) & \mapsto & (z,s=t\prod_{i\in I}\alpha_i(z)^{N_i/m}) \hfill \\ 
   \end{matrix} $$
   is an isomorphism in case $?_\epsilon$ is $=1$ or $=-1$, and induces an isomorphism between the associate double covers 
$\cal R^\epsilon_U =\{(x,t,y)\in (E^0_I\cap U)\times \R\times \R; t^m\cdot u(x)\cdot y^2= \eta(\epsilon) \} $ and 
$\cal R^{'\epsilon}_U=\{(z,s,w)\in (E_I^0\cap U)\times \R\times \R; s^m\cdot v(z)\cdot w^2=\eta(\epsilon)\} $, with $\eta(\epsilon)=1$ when $\epsilon$ is the symbol $>$
and $\eta(\epsilon)=-1$ when $\epsilon$ is the symbol $<$, the induced isomorphism simply being  
$$  \begin{matrix}  
     & \hfill \cal R^\epsilon_U & \to & 
\cal R^{'\epsilon}_U \hfill\\
   & \hfill (x,t,y) & \mapsto & (z,s, w=y ). \hfill \\ 
   \end{matrix}$$
Also notice that $\widetilde E_I^{0,\epsilon}$ is a constructible set when
   $\epsilon$ is $-1$ or $1$ and a semialgebraic set with explicit description 
over the constructible set $E_I^0$ when $\epsilon$ is $<$ or $>$. 
   \end{remark}

We can thus
define the class $[\widetilde E_I^{0,\epsilon}]\in \chi(K_0(BSA_{\R}))$ as follows. Choosing a finite covering $(U_l)_{l\in L}$ of $M$ by affine open subvarieties $U_l$, for $l\in L$, we set
$$[\widetilde E_I^{0,\epsilon}]=\sum _{S\subset L} (-1)^{|S|+1}[R^\epsilon_{\cap_{s\in S}U_{s}}].$$
The class $[\widetilde E_I^{0,\epsilon}]$ does not depend on the choice of the covering thanks to Remark \ref{gluing} and the algebraic additivity in $K_0(BSA_{\R})$.

With this notation one can give the expression of $Z_f^\epsilon(T)$ in terms of 
$[\widetilde E^{0,\epsilon}_I]$, as, 
for instance, in \cite{DL1}, \cite{DL2}, \cite{DL4},  \cite{Loo}, essentially 
using the Kontsevitch change of variables formula in motivic integration (\cite{Kon}, 
\cite{DL2} for instance).

%%%%%%%%%%%%%%%%%%%%%%%%%%%%%%%%%
\begin{theorem}\label{Zeta function} 
%%%%%%%%%%%%%%%%%%%%%%%%%%%%%%%%%
With the notation above, one has  
$$Z_f^\epsilon(T)= \sum_{I\cap \cal K \not= \emptyset}(\L-1)^{\vert I\vert -1 }[\widetilde E^{0,\epsilon}_I ]\prod_{i\in I}
\frac{\L^{-\nu_i}T^{N_i}}{1-\L^{-\nu_i}T^{N_i}}$$
for $\epsilon$ being $-1,1,>$ or $<$.
\end{theorem}

\begin{remark} Classically, the right hand side of equality 
of Theorem 
\ref{Zeta function} does not depend, as a formal series
in $(K_0(\Var_\R)\otimes \Z[\frac{1}{2}])[\L^{-1}][[T]]$, on the 
choice of the resolution $\sigma$, 
as the definition of $Z_f^\epsilon(T)$ does not depend itself on 
any choice of resolution.
\end{remark}
   
To prove this theorem, we first start with a lemma that needs the following notation. We denote by 
$$\sigma_*: \cal L(M,\sigma^{-1}(0))\to \cal L(\R^d,0),$$ and for $n\in \N$, by 
$$\sigma_{n,*} : \cal L_n(M,\sigma^{-1}(0))\to \cal L_n(\R^d,0)$$ 
the natural mappings induced by $\sigma : (M,\sigma^{-1}(0)) \to (\R^d,0)$. Let 
$$Y_{n,f}^\epsilon=\pi_n^{-1}(X_{n,f}^\epsilon).$$
Then $Y_{n,f\circ \sigma}^\epsilon= \{\gamma\in \cal L(M,\sigma^{-1}(0)); \ 
f(\sigma(\pi_n(\gamma)))(t)=at^n+\cdots, \ a \ ?_\epsilon \}$, where $?_\epsilon$ is $=-1$, $=1$, $>0$ or $<0$ in case $\epsilon$ is $-1,1, >$ or 
$< $, and note also that $Y_{n,f\circ \sigma}^\epsilon = \sigma_*^{-1}(Y_{n,f}^\epsilon)$.
Finally for $e\ge 1$, let 
$$\Delta_e=\{ \gamma\in \cal L (M,\sigma^{-1}(0)); \ mult_t \ (\jac \ \sigma)(\gamma(t))=e \} 
\hbox{ and } Y_{e,n,f\circ \sigma}^\epsilon=Y_{n,f\circ \sigma}^\epsilon\cap \Delta_e.$$

\begin{lem}\label{lemme calculatoire} 
%%%%%%%%%%%%%%%%%%%%%%%%%%%%%%%%%
With the notation above, there exists $c\in \N$ such that 
$$ Z_f^\epsilon(T)
= \displaystyle \L^d\sum_{n\ge 1} T^n
 \sum_{e\le cn} \L^{-e}\sum_{I\not=\emptyset}\L^{-(n+1)d} [{\cal L}_n(M,E_I^0\cap \sigma^{-1}(0)) \cap \pi_n(\Delta_e) 
\cap X_{n,f\circ \sigma}^\epsilon].$$ 
\end{lem}
%%%%%%%%%%%%%%%%%%%%%%%%%%%%%%%%%%%%%%%%%%%%%%
\begin{proof}
%%%%%%%%%%%%%%%%%%%%%%%%%%%%%%%%%%%%%%%%%%%%%%
As usual in motivic integration, the class of the cylinder $Y_{n,f}^\epsilon=\pi_n^{-1}(X_{n,f}^\epsilon)$, $n\ge 1$, is 
an element of $(K_0(\Var_\R)\otimes \Z[{1\over 2}])[\L^{-1}]$, 
the localization of the ring $K_0(\Var_\R)\otimes \Z[{1\over 2}]$ with respect to 
the multiplicative set generated by $\L$, and defined
by $[Y_{n,f}^\epsilon]:=\L^{-(n+1)d}[X_{n,f}^\epsilon]$, since the truncation morphisms $\pi_{k+1,k}:\cal L_{k+1}(\R^d,0)\to \cal L_k(\R^d,0)$, $k\ge 1$, are locally trivial 
fibrations with fibre $\R^d$. Hence $Z_f^\epsilon(T)= \displaystyle \L^d\sum_{n\ge 1} [Y_{n,f}^\epsilon]T^n$.

Take now $\gamma \in \sigma_*^{-1}(Y_{n,f}^\epsilon)$, 
and let $I\subset \cal J$ such that $\gamma(0)\in E_I^0$. In some neighbourhood of 
$E_I^0$, one has coordinates such that $f\circ \sigma (x)=u(x)\prod_{i\in I}x_i^{N_i}$ and $\jac(\sigma)(x)=v(x)\prod_{i\in I}x_i^{\nu_i-1}$, with $u$ and
$v$ units.
If one denotes $\gamma=(\gamma_1, \cdots, \gamma_d)$ in these coordinates, with $k_i$ the multiplicity of $\gamma_i$ at $0$ for $i\in I$, 
then we have 
$mult_t(f\circ \sigma\circ \gamma(t))=\sum_{i\in I} k_i N_i=n$. Now 
$$mult_t(\jac \sigma)(\gamma(t))=\sum_{i\in I}k_i(\nu_i-1)\le \max_{i\in I} (\frac{\nu_i-1}{N_i})\sum_{i\in I}N_ik_i
=\max_{i\in I} (\frac{\nu_i-1}{N_i})n.$$  
Therefore if one sets $c= \max_{i\in I} (\frac{\nu_i-1}{N_i})$, one has 
$$Y_{n,f\circ \sigma}^\epsilon=\bigcup_{e\ge 1}Y^\epsilon_{e,n,f\circ \sigma} =\bigcup_{1\le e\le cn} Y^\epsilon_{e,n,f\circ \sigma}, $$
as disjoint unions. Now we can apply the change of variables theorem (see \cite{DL2}, \cite{Kon}) to compute $[Y^\epsilon_{n,f}]$ in terms of 
$[Y^\epsilon_{e,n,f\circ \sigma}]$: 
$$ [Y^\epsilon_{n,f}]= \sum_{e\le cn} \L^{-e}[Y^\epsilon_{e,n,f\circ \sigma}], $$
and summing over the subsets $I$ of $\cal J$, as $Y^\epsilon_{e,n,f\circ \sigma}$ is the disjoint union
 
$$\bigcup_{I\not =\emptyset} Y^\epsilon_{e,n,f\circ \sigma}
 \cap \pi_0^{-1}(E_I^0\cap \sigma^{-1}(0)),$$
we obtain  
$$Z_f^\epsilon(T)= \displaystyle \L^d\sum_{n\ge 1} [Y_{n,f}^\epsilon]T^n=\displaystyle \L^d\sum_{n\ge 1} T^n
 \sum_{e\le cn} \L^{-e}\sum_{I\not=\emptyset}
 [Y^\epsilon_{e,n,f\circ \sigma} \cap \pi_0^{-1}
(E_I^0\cap \sigma^{-1}(0))]$$
$$=\displaystyle \L^d\sum_{n\ge 1} T^n
 \sum_{e\le cn} \L^{-e}\sum_{I\not=\emptyset} \L^{-(n+1)d}
[\pi_n(Y^\epsilon_{e,n,f\circ \sigma} \cap \pi_0^{-1}
(E_I^0\cap \sigma^{-1}(0)))]= $$
$$=\displaystyle \L^d\sum_{n\ge 1} T^n
 \sum_{e\le cn} \L^{-e}\sum_{I\not=\emptyset}\L^{-(n+1)d} [{\cal L}_n(M,E_I^0\cap \sigma^{-1}(0)) \cap \pi_n(\Delta_e) 
\cap X_{n,f\circ \sigma}^\epsilon].$$
\end{proof}
\vskip5mm
%%%%%%%%%%%%%%%%%%%%%%%%%%%%%%%%%%%%%%%%%%%%%
\begin{proof}[Proof of Theorem \ref{Zeta function} ]   
%%%%%%%%%%%%%%%%%%%%%%%%%%%%%%%%%%%%%%%%%%%%%
Considering the expression of $Z_f^\epsilon(T)$ given by Lemma \ref{lemme calculatoire},
we have to compute the class of $[{\cal L}_n(M,E_I^0\cap \sigma^{-1}(0)) \cap \pi_n(\Delta_e) 
\cap X_{n,f\circ \sigma}^\epsilon]$. For this we notice that on 
some neighbourhood $U$ of the end point 
$\gamma(0)\in E_I^0\cap \sigma^{-1}(0)$, 
one has coordinates such that 
$$ f\circ \sigma (x)=u(x)\prod_{i\in I}x_i^{N_i}   \hbox{ and } 
 \jac(\sigma)(x)=v(x)\prod_{i\in I}x_i^{\nu_i-1},$$
with $u$ and $v$ units.
As a consequence ${\cal L}_n(M,E_I^0\cap U\cap \sigma^{-1}(0)) \cap \pi_n(\Delta_e) \cap X_{n,f\circ \sigma}^\epsilon$ is isomorphic to 
$$\{\gamma\in {\cal L}_n(M,\sigma^{-1}(0)); \gamma(0)\in E_I^0\cap U\cap \sigma^{-1}(0), \sum_{i\in I}N_i k_i=n, \sum_{i\in I}
k_i(\nu_i-1)=e,$$
$$ f\circ \sigma (\gamma(t))=at^n+\cdots, a \ ?_\epsilon \},$$
where $?_\epsilon$ is $=-1$, $=1$, $>0$ or $<0$ in case $\epsilon$ is $-1,1, >$ or 
$< $ and $k_i$ is the multiplicity of $\gamma_i$ for $i\in I$.
Now denoting by $A(I,n,e)$ the set 
$$A(I,n,e):=\{k=(k_1, \cdots,k_d)\in \N^d ; \sum_{i\in I}N_i k_i=n, \sum_{i\in I} k_i(\nu_i-1)=e \},$$ 
 and identifying  for simplicity
$x$ and $((x_i)_{i\not\in I},(x_i)_{i\in I})$,
the set 
$${\cal L}_n(M,E_I^0\cap U\cap \sigma^{-1}(0)) \cap \pi_n(\Delta_e) \cap X_{n,f\circ \sigma}^\epsilon$$
 is isomorphic to the product
$$(\R^n)^{d-\vert I\vert}\times \bigcup_{k\in A(I,n,e)} \{ x\in (E_I^0\cap U \cap \sigma^{-1}(0))\times (\R^*)^{\vert I\vert }; 
u((x_i)_{i\not\in I},0) \prod_{i\in I} x_i^{N_i}\ ?_\epsilon\} 
\times \prod_{i\in I}(\R^{n-k_i})$$
Indeed, denoting $\gamma=(\gamma_1, \dots, \gamma_d)$ by
$\gamma_i(t)=a_{i,0}+\cdots + a_{i,n}t^n$ for $i\not \in I$
and $\gamma_i(t)=a_{i,k_i}t^{k_i}+\cdots + a_{i,n} t^n$ for $i\in I$, an arc of $\cal L_n(M,E_I^0\cap U \cap \sigma^{-1}(0))$, the first 
factor of the product comes from the free choice of the coefficients $a_{i,j}$, $i\not\in I$, $j=1, \cdots, n$, the last factor of the product
comes from the free choice of the coefficients $a_{i,j}$, $i\in I$, $j= k_i+1, \dots, n $
and the middle factor of the product comes from the choice of the coefficients $a_{i,0}\in E_I^0\cap U \cap \sigma^{-1}(0)$, $i\not\in I$ and from 
the choice of the coefficients $a_{i,k_i}$, $i\in I$, subject to $f\circ \sigma (\gamma(t))=u(\gamma(t))\prod_{i\in I}\gamma^{N_i}_i(t)=
u((a_{i,0})_{i\not\in I},0) (\prod_{i\in I}a_{i,k_i}^{N_i})t^n+\cdots=at^n+\cdots, a\ ?_\epsilon$.
\vskip2mm
We now choose $n_i\in \Z$ such that $\sum_{i\in I}n_iN_i=m=gcd_{i\in I}(N_i)$ and consider the two semialgebraic sets
$$W_U^\epsilon = \{ x\in (E_I^0\cap U \cap \sigma^{-1}(0))\times (\R^*)^{\vert I\vert }; 
u((x_i)_{i\not\in I},0) \prod_{i\in I} x_i^{N_i}\ ?_\epsilon\}$$
and 
$$W_U^{'\epsilon}= \{ (x',t)\in (E_I^0\cap U \cap \sigma^{-1}(0))\times (\R^*)^{\vert I\vert }\times \R^*; u((x'_i)_{i\not\in I},0)t^m\ ?_\epsilon, \
\prod_{i\in I}x_i'^{N_i/m}=1 \}, $$
where $?_\epsilon$ is $=-1$, $=1$, $>0$ or $<0$ in case $\epsilon$ is $-1,1, >$ or 
$< $.
In case $?_\epsilon=1$ or $?_\epsilon=-1$, the mapping 
$$ \begin{matrix} &W^{'\epsilon}_U &\to &W^\epsilon_U \hfill  \\
                  &(x',t) &\mapsto &x=((x'_i)_{i\not\in I}, (t^{n_i}x'_i)_{i\in I} )\\
\end{matrix}$$
is an isomorphism of inverse 
$$ \begin{matrix} &W_U^\epsilon&\to  &W_U^{'\epsilon}\hfill  \\
                  &x &\mapsto &(x'=((x_i)_{i\not\in I}, 
((\prod_{\ell\in I} x_\ell^{N_\ell/m})^{-n_i}x_i)_{i\in I}), 
t=\prod_{\ell\in I} x_\ell^{N_\ell/m} ). \\
\end{matrix}$$
 In the semialgebraic case, this isomorphism induces a natural  isomorphism
on the double-covers $\cal W_U^\epsilon$ and $\cal W_U^{'\epsilon}$
associated to $ W_U^\epsilon$ and $W_U^{'\epsilon}$ and defined by  
$$\cal W_U^\epsilon=\{(x,y)\in (E_I^0\cap U \cap \sigma^{-1}(0))\times
 (\R^*)^{\vert I\vert }\times \R; 
 y^2u((x'_i)_{i\not\in I},0)\prod_{i\in I} x_i^{N_i}
=\eta(\epsilon) \}$$ and 
$$\cal W_U^{'\epsilon}=\{ (x,t,w)\in (E_I^0\cap U \cap \sigma^{-1}(0))\times (\R^*)^{\vert I\vert }\times \R^*\times \R;$$
$$w^2u((x'_i)_{i\not\in I},0)t^m =\eta(\epsilon), \
\prod_{i\in I}x_i'^{N_i/m}=1\},$$
where $\eta(\epsilon)=1$ when $\epsilon$ is the symbol $>$ and
$\eta(\epsilon)=-1$ when $\epsilon$ is the symbol $<$. In consequence, $[W_U^\epsilon]=[W_U^{'\epsilon}]$ in the algebraic
case ($\epsilon=-1$ or $1$) as well as in the semialgebraic case ($\epsilon=<$ or $>$) considering our realization formula for basic semialgebraic
formulas in 
$ K_0(\Var_\R)\otimes \Z[\frac{1}{2}]$.
Now we observe in the case where $\epsilon$ is $-1$ or $1$ that $W_U^{'\epsilon}$ is isomorphic to 
$R^\epsilon_U \times (\R^*)^{\vert I\vert -1}$ (see \cite{DL4}, Lemma 2.5) whereas in the case where 
$\epsilon$ is $<$ or $>$, we obtain that the class of $W_U^{'\epsilon}$ is equal to the class of 
$R^\epsilon_U \times (\R^*)^{\vert I\vert -1}$, considering again the double coverings associated to the
basic semialgebraic formulas defining these two sets. 

We finally obtain  
$$[{\cal L}_n(M,E_I^0\cap \sigma^{-1}(0)) \cap \pi_n(\Delta_e) 
\cap X_{n,f\circ \sigma}^\epsilon]=\displaystyle\sum_{k\in A(I,n,e)}
\L^{nd-\sum_{i\in I}k_i}[W_U^{'\epsilon}]=$$
$$\displaystyle\sum_{k\in A(I,n,e)}
 \L^{nd-\sum_{i\in I}k_i}\times[R_U^{\epsilon}]\times 
(\L-1)^{\vert I\vert -1}.$$
Summing over the charts $U$, the expression of $Z_f^\epsilon(T)$ given by Lemma \ref{lemme calculatoire} is now 
$$ Z_f^\epsilon(T)=\displaystyle \sum_{I\cap \cal K \not= \emptyset} \L^d\sum_{n\ge 1} T^n
 \sum_{e\le cn} \L^{-e}(\L-1)^{\vert I \vert-1 }\L^{-(n+1)d} 
[{\widetilde E}_I^{0,\epsilon}]\displaystyle\sum_{k\in A(I,n,e)}  \L^{nd-\sum_{i\in I}k_i}$$
$$=\displaystyle \sum_{I\cap \cal K \not= \emptyset} (\L-1)^{\vert I \vert-1 }
[{\widetilde E}_I^{0,\epsilon} ]
\sum_{n\ge 1} T^n
\sum_{e\le cn}\sum_{k\in A(I,n,e)}  \L^{-e-\sum_{i\in I}k_i}$$
 Noticing that the $(k_i)_{i\in I}$'s such that $k=((k_i)_{i\not\in I}),
(k_i)_{i\in I})\in \displaystyle \bigcup_{e\le cn, n\ge 1} A(I,n,e)$ are
in bijection with $\N^{*\vert I \vert }$, we have 
$$Z_f^\epsilon(T)=
\displaystyle \sum_{I\cap \cal K\not= \emptyset} (\L-1)^{\vert I \vert-1 }
[{\widetilde E}_I^{0,\epsilon}]
\sum_{(k_i)_{i\in I}\in \N^{\vert I \vert }}  
\prod_{i\in I}(\L^{-\nu_i}T^{N_i})^{k_i}$$
$$= \displaystyle \sum_{I\cap \cal K\not= \emptyset} (\L-1)^{\vert I \vert-1 }
[{\widetilde E}_I^{0,\epsilon}]
  \prod_{i\in I}\frac{\L^{-\nu_i}T^{N_i}}{1-\L^{-\nu_i}T^{N_i}}. $$
\end{proof}

%%%%%%%%%%%%%%%%%%%%%%%%%%%%%%%%%%%%%%%%%%%%%%%%%%%%%%%%%%%%%%%%%%%
\subsection{Motivic real Milnor fibres and their realizations.}
%%%%%%%%%%%%%%%%%%%%%%%%%%%%%%%%%%%%%%%%%%%%%%%%%%%%%%%%%%%%%%%%%%%
We can now define a motivic real Milnor fibre by taking the constant term
of the rational function $Z_f^\epsilon(T)$ viewed as a power series in $T^{-1}$. 
This process formally consists in letting $T$ going to $\infty$ in the rational 
expression of  $Z_f^\epsilon(T)$ given by Theorem \ref{Zeta function} and using 
the usual computation rules as in the convergent case (see for instance \cite{DL1}, \cite{DL4}).  

\begin{definition}\label{real D-L} 
Let $f:\R^d\to \R$ be a polynomial function and $\epsilon$ be one of the symbols $naive, 1, -1, >$ or $<$. 
Consider a resolution of $f$ as above and let us adopt the same notation $(E_I^0)_I$ for the stratification of the exceptional 
divisor of this resolution, leading to the notation $\widetilde E_I^{0,\epsilon}$. 
The real motivic 
Milnor $\epsilon$-fibre $S^\epsilon_f$ of $f$ is defined as (see \cite{DL4} for the complex case) 
$$ S^\epsilon_f:=-\lim_{T\to \infty} Z^\epsilon_f(T):=\displaystyle
-\sum_{I\cap \cal K\not= \emptyset} (-1)^{\vert I \vert}
[{\widetilde E}_I^{0,\epsilon}](\L-1)^{\vert I \vert-1 }
\in K_0(\Var_\R)\otimes \Z[\frac{1}{2}].$$
It does not depend on the choice of the resolution $\sigma$.
\end{definition}
For $\epsilon$ being the symbol $1$ for instance, we have $S_f^1\in K_0(\Var_\R)$.
We can consider, first in the complex case, the realization 
of $S_f^1$ via the Euler-Poincar\'e characteristic ring morphism $\chi_c : K_0(\Var_\C)\to \Z$. Note that in the complex case, the Euler characteristics with and without compact supports are equal. For $f:\C^d\to \C$, since $\chi_c(\L-1)=0$, we obtain 

$$\chi_c(S_f^1)=\displaystyle\sum_{\vert I\vert=1, I\subset \cal K} 
\chi_c({\widetilde E}_I^{0,1}) =\displaystyle\sum_{\vert I\vert=1, I\subset \cal K}  
N_I\cdot \chi_c(E_I^0\cap \sigma^{-1}(0)).$$  
Now denoting by
$F$ the set-theoretic Milnor fibre of the fibration $f_{\vert B(0,\alpha)\cap f^{-1}(D^\times_\eta)}:   
B(0,\alpha)\cap f^{-1}(D^\times_\eta) 
\to D_\eta^\times $, with $B(0,\alpha)$ the open ball in $\C^d$ of radius $\alpha$ 
centred at 
$0$, $D_\eta$ the disc in $\C$ 
of radius $\eta$ centred at $0$ and $D_\eta^\times=D_\eta\setminus\{0\} $, with 
$0<\eta\ll \alpha\ll 1$, 
comparing the above expression $\chi_c(S_f^1)=\displaystyle
\sum_{\vert I\vert=1, I\subset \cal K} 
N_I\cdot \chi_c(E_I^0)$ with the following A'Campo formula of \cite{ACA}  
 for the first Lefschetz number of the iterates of the monodromy 
$M:H^*(F,\C)\to H^*(F,\C)$ of $f$, 
that is for the Euler-Poincar\'e characteristic of the fibre $F$:
 
$$ \chi_c(F)=\displaystyle\sum_{\vert I\vert=1, I\subset \cal K} N_I\cdot 
\chi_c(E_I^0\cap \sigma^{-1}(0)) $$
we   simply observe  that 
$$\chi_c(S_f^1)= \chi_c(F).$$
The closure $f^{-1}(c)\cap \bar B(0,\alpha)$, $0<\vert c\vert \ll \alpha \ll 1$, of the 
Milnor fibre $F$ being denoted by $\bar F$ and the boundary of $\bar F$
being the odd dimensional compact 
manifold $f^{-1}(c)\cap S(0,\alpha)$, 
$\chi_c(f^{-1}(c)\cap S(0,\alpha) )=0$ and we finally have
$$\chi_c(S_f^1)=\chi_c(F)=\chi_c(\bar F).$$

\begin{remark} There is {\sl a priori} no hint in the definition of $Z_f^\epsilon(T)$
that the opposite of the constant term $S_f^1$ of the power series in $T^{-1}$ induced by the 
rationality of $Z_f^\epsilon(T)$
could be the motivic version of the Milnor fibre of $f$ (as well
as, for instance, there is no evident hint that the expression of $Z_f^\epsilon$ in Theorem \ref{Zeta function}
does not depend on the resolution $\sigma$).  
As mentionned above, in the complex case, we just observe that the   
expression of $\chi_c(S_f^1)$ is the expression of $\chi_c(F)$ provided 
by the A'Campo formula. Exactly in the same way there is no 
{\sl a priori} reason for $\chi_c(S_f^\epsilon)$, regarding the definition of $Z_f^\epsilon$, to be so acurately related to 
the topology of $f^{-1}(\epsilon \vert c \vert)\cap B(0,\alpha)$.  Nevertheless we prove
that it is actually the case (Theorem \ref{Milnor}).
\end{remark}
In order to establish this result we start hereafter by a 
geometrical proof of the formula in the complex
case (compare with \cite{ACA} where only $\Lambda(M^0)$ is considered, $M^k$ being 
the $k$th iterate of the monodromy $M:H^*(F,\C)\to H^*(F,\C)$ of $f$). 
We will then extend to the reals this computational proof in
 the proof of Theorem \ref{Milnor}, allowing us   interpret the complex proof 
as the first complexity level of its real extension.

\begin{remark} Note that in the complex case a proof of the fact that $\Lambda(M^k)=\chi_c(X^1_{k,f})$, 
for $k\ge 1$, is given in \cite{HL}
without the help of resolution of singularities, that is to say without help of 
A'Campo's formulas (see Theorem 1.1.1 of \cite{HL}). As a direct corollary it is thus proved that  $\chi_c(S^1_f)=\chi_c(F)$
in the complex case, without using 
A'Campo formulas.
\end{remark}

\vskip5mm
\noindent
{\bf Realization of the complex motivic Milnor fibre under $\chi_c$. } 
The fibre $F=\{ f=c\}\cap B(0,\alpha)$ is homeomorphic to the fibre 
$\cal F=\{f\circ \sigma = c\}\cap \sigma^{-1}(B(0,\alpha))$, with $\sigma^{-1}(S(0,\alpha))$ viewed as the boundary of a 
tubular neighbourhood
of  $\sigma^{-1}(0)=\bigcup_{E_J^0\subset \sigma^{-1}(0)}E_J^0$, 
keeping the same notation $(E_J^0)_J$ as before for the natural stratification of the strict transform $\sigma^{-1}(\{f=0\})$ of  $f=0$. 
Now the formula may be established for 
$\cal F$ in some chart of $M\cap \sigma^{-1}(B(0,\alpha))$, by additivity.
In such a chart, where $f\circ \sigma$ is normal crossing, we consider 
\begin{enumerate}
\item[-] the set $E_J=\bigcap_{i\in J}E_i \subset \sigma^{-1}(0)$,
given by $x_i=0$, $ i\in J$, 
\item[-] a closed small enough tubular neighbourhood $V_J$ in $M$ of  
$\bigcup_{J\subset K, K\not= J} E^0_K$, that is 
a tubular neighbourhood of all the $E^0_K$'s bounding $E^0_J$, such that
$E_J^0\setminus V_J$ is homeomorphic to $E_J^0$,
\item[-] and $\pi_J$  the projection onto $E_J$ along the
$x_j$'s coordinates, for $j\in J$. 
\item[-] an open  neighbourhood $\cal E_J$ 
of $E_J^0\setminus V_J$ in $\sigma^{-1}(B(0,\alpha)) $ given by 
$\pi_J^{-1}(E_J^0\setminus V_J), \vert x_j\vert \le \eta_J$,  
$j\in J $, with $\eta_J>0$ small enough,

\end{enumerate}

%%%%%%%%%%%%%%%%
\begin{remark}\label{Thom}
%%%%%%%%%%%%%%%%
For $I=\{i \}$, we remark that $\cal F \cap \cal E_I$ is homeomorphic
to $N_i$ copies of $E_I^0\cap \cal E_I$, and thus to $N_i$ copies of 
$E_I^0$. Indeed, assuming $f\circ \sigma = u(x)x_i^{N_i}$
in $\cal E_I$, we observe that the family $(f_t)_{t\in [0,1]}$, with 
$f_t=u((x_j)_{j\not\in I}, t\cdot x_i)  x_i^{N_i}-c$, 
has homeomorphic fibres $\{f_t=0\}\cap \cal E_J$, 
$t\in [0,1]$, by Thom's isotopy lemma, since
$$\frac{\partial f_t}{\partial x_i}(x)=t\frac{\partial u}{\partial x_i}(x)
 x_i^{N_i}+
u(x) x_i^{N_i-1}=0, $$
would imply  $\displaystyle t\frac{\partial u}{\partial x_i}(x)x_i+
u(x)=0$. But the first term in this sum goes to $0$ as $x_i$ goes to $0$, since the 
derivatives of $u$ are
bounded on the compact $adh(\cal E_I)$, although the norm of the 
second term is bounded from below on $\cal E_I$ by a non zero constant, since $u$ is a unit. 
Finally, as
$\{ f_1=0\}\cap \cal E_I$ is homeomorphic to $\{ f_0=0\}\cap \cal E_I$ and 
$\{ f_0=0\}\cap \cal E_I$ 
is a $N_i$-graph 
over $E_I^0\cap \cal E_I$, $\cal F\cap \cal E_I$ is homeomorphic to $N_i$ copies of $E_I^0$.
\end{remark}
%%%%%%%%%%%%%

By this remark, $\cal F$ covers maximal dimensional stratum $E_I^0$,
$\vert I \vert=1$, $I \subset \cal K$, with $N_i$ copies of a leaf $\cal F_I$ of $\cal F$.
To be more accurate, with the notation introduced above, 
$\cal F_I$ covers the neighbourhood $E_I^0 \cap \cal E_I$ of $E_I^0\setminus V_I$.  
Moreover  the $\cal F_I$'s overlap in $\cal F$
over the open set $E_J^0 \cap \cal E_J$ of the strata $E^0_J$ that bound
the $E^0_I$'s,  for $\vert I \vert =1$, $\vert J\vert =2$ 
and $I \subset J$, in bundles over the $E_J^0 \cap \cal E_J$'s of fibre $\C^*$. 
Those sub-leaves 
$\cal F_J$ of $\cal F$ 
overlap in turn over the open $E_Q^0 \cap \cal E_Q$ of the strata $E_Q^0$, 
$\vert Q\vert=3, J\subset Q $,
that bound the $E_J^0$'s, in bundles over the $E_Q^0 \cap \cal E_Q$'s of fibres
$(\C^*)^2$ and so forth... 
For instance when $f\circ \sigma = u(x)\prod_{i\in I}x_i^{N_i}$ in  
$\cal E_I$, $I=\{i\}$, 
and 
$f\circ \sigma = v(x) x_i^{N_i}x_j^{N_j}$ in  $\cal E_J$, $J=\{i,j\}$, the $N_i$ leaves 
$\cal F_I$, homeomorphic to the $N_i$ copies $x_i^{N_i}=c /u(x)$ of $E_I^0$, overlap 
over $E_J^0 \cap \cal E_J$
in sub-leaves $\cal F_J$ of $\cal F_I$, given by $v(x)x_i^{N_i}x_j^{N_j}=c$, fibering over $E_J^0$ with 
fibre $GCD(\{N_i,N_j\})$ copies of $(\C^*)^{\vert J\vert -1}$ and so forth... (see figure 1).

\vskip1,0cm
\hskip1,5cm \includegraphics[height=8cm]{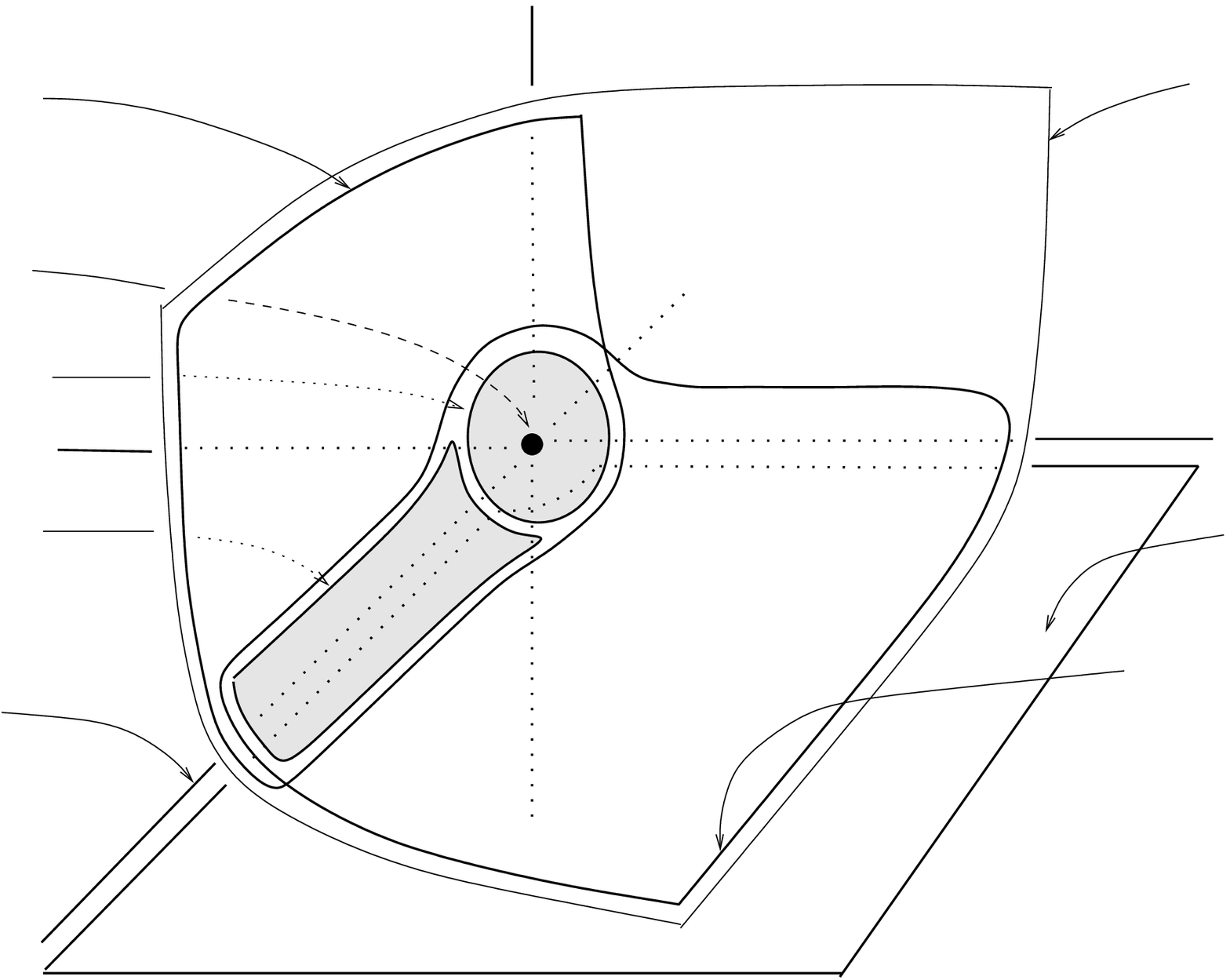}
\vskip-7,7cm	 
\hskip11,4cm$f\circ \sigma=c$
\vskip-0,4cm	 
\hskip1,0cm $\cal F_{I'}$
\vskip0,8cm	 
\hskip0,9cm $E^0_K$
\vskip0,5cm	 
\hskip1,0cm $\cal F_K$
\vskip0,8cm	 
\hskip1,1cm $\cal F_J$
\vskip-0,5cm	 
\hskip11,7cm $E^0_I$
\vskip0,7cm	 
\hskip10,9cm $\cal F_I$
\vskip-0,2cm	 
\hskip0,9cm $E^0_J $

\vskip3,0cm	
\centerline{figure 1}
\vskip1cm	 

\begin{remark}\label{retract}
Note that the topology of $\cal F=\{f\circ \sigma =c\} \cap \sigma^{-1}(B(0,\alpha))$ 
is the same as the topology of 
$\bigcup_{J\cap \cal K \not=0}\cal F_J$ (that is the topology of $\cal F$ above the strata 
$E^0_J$ of $\sigma^{-1}(0)$) since
the retraction of $\cal F$ onto $\bigcup_{J\cap \cal K \not=\emptyset} \cal F_J$, 
as $\alpha$ goes to $0$, induces
a homeomorphism from $\cal F$ to $\bigcup_{J\cap \cal K \not=\emptyset} \cal F_J$. 
\end{remark}

From Remark \ref{retract}, by additivity, it follows that the Euler-Poincar\'e characteristic of $\cal F$ (in our chart) is the sum 
$$\displaystyle\sum_{\vert I\vert=1, I\subset \cal K} 
N_I\cdot \chi_c(E_I^0\cap \sigma^{-1}(0)) + L,\eqno(*)$$ where $L$ is some $\Z$-linear 
combination of 
Euler-Poincar\'e characteristics of bundles
over the open sets $E_J\cap \cal E_J^0$, $\vert J \vert>1 $, of fibre a power of tori $\C^*$. Now the A'Campo formula 
$$ \chi_c(F)=\displaystyle\sum_{\vert I\vert=1, I\subset \cal K} 
N_I\cdot \chi_c(E_I^0\cap \sigma^{-1}(0))$$
for
the Milnor number  follows from the fact that $\chi_c(\C^*)=0$ implies $L=0$.

\vskip5mm
\noindent
{\bf Realization of the real motivic Milnor fibres under $\chi_c$. } 
The partial covering of $\cal F$ by the pieces $\cal F_J$, for 
$J\cap \cal K \not= \emptyset$, 
over the strata of the stratification $(E^0_J)_{J\cap \cal K \not= \emptyset}$ of 
  $\sigma^{-1}(0)$  allows us to compute the Euler-Poincar\'e characteristic 
of the Milnor fibre $\cal F$ in terms of the Euler-Poincar\'e characteristic of 
the strata $E_J^0$, in the complex as well as in the real case. In the complex case, as noted above, for $J$ with $\vert J\vert >1$,
one has $\chi_c(\cal F_J)=0$. This cancellation provides a quite simple formula for
 $\chi_c(F)$: only the strata of the maximal 
dimension   of the divisor $\sigma^{-1}(0)$ appear in this formula, as expected 
from the A'Campo formula. 

In the real case one does not have such cancellations: on one hand 
the expression of $\chi_c(F)$ 
in terms of $\chi_c(\widetilde E_J^{0,\epsilon})$ is no more trivial
(the remaining term $L$ of equation $(*)$ is not zero and consequently terms 
$\chi_c(\widetilde E_J^{0, \epsilon})$, for $\vert J\vert >1 $ and 
$E_j \cap \sigma^{-1}(0)\not=\emptyset$, appear), and on the other hand the 
expression
of $\chi_c(S_f^\epsilon)$ given by the real Denef-Loeser formula in Definition
\ref{real D-L} have terms $2^{\vert J \vert -1}\chi_c(\widetilde E_J^{0,\epsilon})$
, for $\vert J\vert >1 $ and $J\cap \cal K \not=\emptyset $
(since $\chi_c(\L-1) =-2$ in the real case).

Nevertheless, in the real case we show that
$\chi_c(S_f^\epsilon)$  is again $\chi_c(\bar F)$, justifying the terminology of  
motivic real semialgebraic Milnor fibre of $f$ at $0$ for $S_f^\epsilon$. The formula 
stated in Theorem \ref{Milnor} below 
is the real analogue of the A'Campo-Denef-Loeser formula for complex hypersurface 
singularities and thus appears as the extension to the 
reals of this complex formula, or, in other words, the complex formula is the notably 
first level 
of complexity of the more general real formula.

\begin{notation}\label{Khim}
Let $f: \R^d\to \R$ be a polynomial function such that $f(0)=0$ and with isolated 
singularity at $0$, that is $\grad f(x)=0$ only for $x=0$ in some open neighbourhood of 
$0$.
 Let $0<\eta \ll \alpha$ be such that the topological type   of 
$f^{-1}(c)\cap B(0,\alpha)$ does 
not depend on $c$ and $\alpha$, for $0<c<\eta$ or for $-\eta<c<0$.

- Let us denote, for $\epsilon\in \{-1,1\}$ and $\epsilon\cdot c>0$, this topological type by $F_\epsilon$, by $\bar F_\epsilon$ the topological type of the closure of the Milnor fibre $F_\epsilon$ and by 
$Lk(f)$ the link $f^{-1}(0)\cap S(0,\alpha)$ of $f$ at the origin. 
We recall that the topology of $Lk(f)$ is the same as the topology of the boundary 
$f^{-1}(c)\cap S(0,\alpha)$ of the Milnor fibre $\bar F_\epsilon$, when $f$ has an
isolated singularity at $0$.

- Let us denote, for $\epsilon\in \{<,>\}$, the topological type
of $f^{-1}(]0,c_\epsilon[)\cap B(0,\alpha)$ by $F_\epsilon$, and 
the topological type of $f^{-1}(]0,c_\epsilon[)\cap \bar B(0,\alpha)$
by $\bar F_\epsilon$,
where $c_<\in ]-\eta, 0[$ and $c_>\in ]0,\eta[$.

- Let us denote, for $\epsilon\in \{<,>\}$, the topological type
of $\{f \ \bar \epsilon \ 0\}\cap S(0,\alpha)$ by $G_\epsilon$, where 
$\bar \epsilon$ is $\le$ when  $\epsilon$ is $<$ and $\bar \epsilon$ is $\ge$
when  $\epsilon$ is $>$. 
\end{notation}

\begin{remark}\label{bord}
When $d$ is odd, $Lk(f)$ is a smooth odd-dimensional submanifold of $\R^d$ and 
consequently $\chi_c(Lk(f))=0$. For 
$\epsilon \in \{-1, 1, <, >\}$, we thus have in this situation,  
$\chi_c(F_\epsilon)=\chi_c(\bar F_\epsilon)$. This is the situation in the complex 
setting. When $d$ is even and for  $\epsilon \in \{-1, 1\}$ since
$\bar F_\epsilon$ is a compact manifold with boundary
$Lk(f)$, one knows that
$$\chi_c(\bar F_\epsilon)=-\chi_c (F_\epsilon)=\frac{1}{2}\chi_c(Lk(f)).$$
For general $d\in \N$ and for $\epsilon \in \{-1, 1, <, >\}$, we thus have 
$$ \chi_c(\bar F_\epsilon)=(-1)^{d+1}\chi_c (F_\epsilon).$$
On the other hand we recall that for $\epsilon \in \{ <, >\}$
$$ \chi_c(G_\epsilon )=\chi_c(\bar F_{\delta_\epsilon}),$$
where $\delta_>$ is $1$ and $\delta_<$ is $-1$ 
(see \cite{Ar}, \cite{Wa}).
\end{remark}

%%%%%%%%%%%%%%%%%%%%%%%%%%%%%%%%%
\begin{theorem}\label{Milnor} 
%%%%%%%%%%%%%%%%%%%%%%%%%%%%%%%%%
 With  notation \ref{Khim}, we have, for $\epsilon\in \{-1, 1, <, >\}$ 
$$ \chi_c(S_f^\epsilon)=\chi_c(\bar F_\epsilon)=(-1)^{d+1}\chi_c (F_\epsilon), $$
and for $\epsilon\in \{<, >\}$
$$ \chi_c(S_f^\epsilon)=-\chi_c(G_\epsilon).$$

\end{theorem}

\begin{proof} Assume first that $\epsilon\in \{-1, 1\}$. We denote
by $\cal F$ the fibre 
$\sigma^{-1}(F_\epsilon)$ and recall that $\cal F$ and $F_\epsilon$ have the same 
topological type.  
Let us denote $\bar{\cal K}$ the set of multi-indices 
$J\subset  \cal I$ such that 
$\bar E_J \cap \sigma^{-1}(0)\not=\emptyset$, with
$\bar E_J$ the closure of $E_J=\bigcap_{i\in J} E_i$. In what follows
only  $J\in\bar{\cal K}$ are concerned, 
since we study the 
local Milnor fibre at $0$. 
The proof consists in the computation 
of the Euler-Poincar\'e characteristic of $\cal F$ using the decomposition of $\cal F$ by the overlapping components $\cal F_I$
introduced  just before figure 1 and illustrated on figure 1. We simply count the number of these overlapping components in the decomposition of $\cal F$ they provide.
 Note that a connected component of $E^0_J$ (still denoted $E_J^0$ for simplicity in the sequel), for $J\subset \cal J$, is covered by 
$n_J:=M_J\cdot 2^{\vert J\vert -1}$ connected components $\cal G$ 
of $\cal F$, where $M_J$ is 
$0$, $1$ or $2$ depending on the fact that the multiplicity $m_J=gcd_{j\in J}(N_j)$ defining $\widetilde E^{0,\epsilon}_J$ is 
odd or even, and on sign condition on $c$ (remember from figure 1 how $E^0_J$ is covered by $\cal F_J$. Here the term covered simply means that one can naturally project the component $\cal F_J$ onto $E^0_J$). 
Note furthermore that $M_J$ is the degree of the covering $\widetilde E_J^{0,\epsilon}$ of $E_J^0$. 
Now expressing a connected component $\cal G$ of $\cal F$ as the union 
$\displaystyle\bigcup_{\vert I \vert=1, \cal F_I\subset \cal G} \cal F_I$, where the 
(connected) leaves $\cal F_I$ 
cover (the open subset $E^0_I\cap \cal E_I^0$ of $E^0_I$ homeomorphic to) 
$E^0_I$,
 and using the additivity of $\chi_c$, one has
that $\chi_c(\cal G)$ is expressed as a sum of characteristics of the overlapping 
connected sub-leaves
$\cal F_J$ of the $\cal F_I$'s, each of them with sign coefficient  
$s_J:=(-1)^{\vert J \vert -1 }$ . 
Note that (a connected component of) 
$E_J^0$ is covered by $n_J$ copies of such 
a $\cal F_J$, coming from the $n_J$ connected components of $\cal F$ above $E_J^0 \cap \cal E_J^0$,
 and that a connected sub-leaf $\cal F_J$ has the topology of  
$(E_J^0\cap \cal E_J^0)\times \R^{\vert J\vert -1}$. We denote by $t_J$ the characteristic 
$t_J:=\chi_c(\R^{\vert J\vert -1})=(-1)^{\vert J\vert -1}$. 
 
With this notation, summing over all the connected components $\cal G$ of $\cal F$,  one 
gets 
$$\chi_c(\cal F)=\sum_{J\in \bar{\cal K}} 
s_J \times  n_J \times \chi_c(E^0_J)\times t_J$$
$$=\sum_{J\in \bar{\cal K}} 
(-1)^{\vert J \vert -1} \times  2^{\vert J\vert -1}M_J \times \chi_c(E^0_J)
\times  (-1)^{\vert J \vert -1}$$
$$=
\sum_{J\in \bar{\cal K}} 2^{\vert J\vert-1}
\chi_c(\widetilde E_J^{0,\epsilon})  $$
$$=\sum_{J\cap\cal K\not=\emptyset} 2^{\vert J\vert-1} \chi_c(\widetilde E_J^{0,\epsilon}) 
+ \sum_{J\cap\cal K=\emptyset, J\in \bar{\cal K}} 2^{\vert J\vert-1} \chi_c(\widetilde E_J^{0,\epsilon}) $$
$$ = \chi_c(S_f^\epsilon)
+\sum_{J\cap\cal K=\emptyset,J\in \bar{\cal K}} 2^{\vert J\vert-1} 
\chi_c(\widetilde E_J^{0,\epsilon})$$
$$ = \chi_c(S_f^\epsilon)
+\chi_c(\bigcup_{J\cap\cal K=\emptyset,J\in \bar{\cal K}} \cal F_J). $$
Note that the sub-leaves $\cal F_J$ for $J\cap \cal K=\emptyset$ and $J\in \bar{\cal K}$ cover
the set $\{f\circ \sigma = c\}\cap \hat S(0,\alpha)$, for $\epsilon\cdot c >0$, 
where $\hat S(0,\alpha)$
is a neighbourhood $\sigma^{-1}(S(0,\alpha)\times ]0,\beta[)$ of 
$\sigma^{-1}(S(0,\alpha))$, with $0<\beta \ll \alpha$.
It follows that 
$$\chi_c(\bigcup_{J\cap\cal K=\emptyset,J\in \bar{\cal K}} \cal F_J)= 
\chi_c(F_\epsilon\cap (S(0,\alpha)\times ]0,\beta[))=\chi_c(Lk(f)\times ]0,\beta[)=
-\chi_c(Lk(f)). $$
We finally obtain
$$ \chi_c(F_\epsilon)=\chi_c(S_f^\epsilon)-\chi_c(Lk(f)),$$
and 
$$ \chi_c(\bar F_\epsilon)= \chi_c(F_\epsilon)+\chi_c(Lk(f))=\chi_c(S_f^\epsilon).$$
This proves the first equality of our statement, the   equality 
$\chi_c(\bar F_\epsilon)=(-1)^{d+1}\chi_c (F_\epsilon)$ being proved
in Remark \ref{bord}.

Assume now that $\epsilon \in \{<,>\}$, and denote  $\delta_<:=-1$ and $\delta_>:=1$, like
in Remark \ref{bord}. 
With this notation 
$\bar F_\epsilon=\bar F_{\delta_\epsilon}\times ]0,1[$, 
and by the formula proved above in the case  $\epsilon\in \{-1,1\}$, we obtain  
$$ \chi_c(\bar F_\epsilon)=\chi_c(\bar F_{\delta_\epsilon})\chi_c(]0,1[)
=-\chi_c(\bar F_{\delta_\epsilon})=-\chi_c(S_f^{\delta_\epsilon})=- 
\sum_{J\cap \cal K \not=\emptyset  } 2^{\vert J\vert-1}
\chi(\widetilde E_J^{0,\delta_\epsilon}).$$
But since $\widetilde E_J^{0,\epsilon}=
\widetilde E_J^{0,\delta_\epsilon}\times \R_+$, it follows that 
$$\chi_c( \bar F_\epsilon)= \sum_{J\cap \cal K \not=\emptyset } 2^{\vert J\vert-1} 
\chi(\widetilde E_J^{0,\delta_\epsilon})\chi_c(\R_+)=
\sum_{J\cap \cal K \not=\emptyset} 2^{\vert J\vert-1} \chi(\widetilde E_J^{0,\epsilon})
=\chi_c(S_f^\epsilon).$$
This proves the first equality of our statement.  The   equality 
$\chi_c(\bar F_\epsilon)=(-1)^{d+1}\chi_c (F_\epsilon)$
is the consequence of $\chi_c(\bar F_\epsilon)=\chi_c(\bar F_{\delta_\epsilon})\chi_c(]0,1[)$,
$\chi_c(  F_\epsilon)=\chi_c(  F_{\delta_\epsilon})\chi_c(]0,1[)$  
 and 
$\chi_c(\bar F_{\delta_\epsilon})=(-1)^{d+1}\chi_c (F_{\delta_\epsilon})$.
 
 To finish,  equality $\chi_c(S_f^\epsilon)=-\chi_c(G_\epsilon)$ comes from the
 equality $ \chi_c(G_\epsilon )=\chi_c(\bar F_{\delta_\epsilon})$ 
recalled in Remark \ref{bord} and from 
$\chi_c(\bar F_\epsilon)=-\chi_c(\bar F_{\delta_\epsilon})$, 
$\chi_c(S_f^\epsilon)=\chi_c(\bar F_\epsilon)$.

\end{proof}

%%%%%%%%%%%%%%
\begin{remark}
%%%%%%%%%%%%%%
As stated in Theorem \ref{Milnor}, the realization via $\chi_c$
of the motivic Milnor fibre $S_f^\epsilon$, for $\epsilon\in \{-1,1,<,>\}$, gives the 
Euler-Poincar\'e characteristic of the corresponding set theoretic semialgebraic  
closed  Milnor fibre $\bar F_\epsilon$. 
Nevertheless it is worth noting that this equality is in general not true at the higher 
level of
 $\chi(K_0[BSA_\R])$. Even computed in $K_0(\Var_\R)\otimes \Z[\frac{1}{2}]$, we may have
 $S_f^\epsilon\not=[A_{f,\epsilon}]$, for a given semialgebraic formula 
$A_{f,\epsilon}$ with real points $\bar F_\epsilon$.
 Let us illustrate this remark by the following quite trivial example.
\end{remark}

\begin{example}
Let us consider the simple case where $f:\R^2\to \R$ is given by $f(x,y)=xy$. 
After one blowing-up  $\sigma: M\to \R^2$ of the origin of $\R^2$, the situation is as required by Theorem \ref{Zeta function}.
We denote by $E_1$ the exceptional divisor $\sigma^{-1}(0)$ (which is isomorphic to $\P_1$) and by 
$E_2,E_3$ the irreducible components of the strict transform $\sigma^{-1}(\{f=0\})$. 
The induced
stratification of $E_1$ is given by $E_{1,2}^0=E_1\cap E_2$, $E_{1,3}^0=E_1\cap E_3$, and
the two connected components ${E}_1^{'0}, {E}_1^{''0}$ of $E_1\setminus (E_2\cup E_3)$.
We consider a chart $(X,Y)$ of $M$ such that $\sigma(X,Y)=(x=Y,y=XY)$. In this chart 
$(f\circ \sigma)(X,Y)=XY^2$. The multiplicity of $f\circ \sigma$ along $E_1$ is $N_1=2$, 
and the multiplicity of $\jac\sigma$ along $E_1$ is $1$, thus $\nu_1=2$. Assuming that 
$ E_1^{'0}$ corresponds to $X>0$ and $ E_1^{''0}$ corresponds to $X<0$,
it follows that 
$$\widetilde E^{'0,\epsilon}_1=\{(X,t); X\in E_1^{'0}, t\in \R, Xt^2?_\epsilon\}
\hbox{ and }  \widetilde E^{''0, \epsilon}_1=\{(X,t); X\in E_1^{''0}, t\in \R, Xt^2?_\epsilon\},$$ 
where $?_\epsilon$ is
$=1$, $=-1$, $>$ or $<0$ in case $\epsilon$ is $1$, $-1$, $>$ or $<$.
In case $\epsilon=1$ we obtain  
$$[\widetilde E^{'0,1}_1]=\L-1 \hbox{ and  } [\widetilde E^{''0,1}_1]=0$$ 
since $\widetilde E^{'0,1}_1$ has a one-to-one projection onto 
 $\{(X,Y); X=0, Y\not=0\}$) and $\widetilde E^{''0,1}_1$ is
empty.
Now in a neighbourhood of $E^0_{1,2}$, $f\circ \sigma (X,Y)=XY^2$, 
giving $N_1=1$, $N_2=2$ and $m=gcd(N_1,N_2)=1$. We also have $\nu_1=2$ and $\nu_2=1$. 
It follows that 
$$\widetilde E^{0,1}_{1,2}=\{(0,t); t\in \R, t=1\} \hbox{ thus }  [\widetilde E^{0,1}_{1,2}]=1.$$
 In the same way, using another chart, one finds 
 $$ \ [\widetilde E^{0,1}_{1,3}]=1.$$
By Theorem \ref{Zeta function} we then have 
$$ Z_f^1(T)=(\L-1)^{1-1}(\L-1)\Big(\frac{\L^{-2}T^2}{1-\L^{-2}T^2}\Big)
+
 2(\L-1)^{2-1}\Big(\frac{\L^{-2}T^2}{1-\L^{-2}T^2}\Big)
\Big(\frac{\L^{-1}T}{1-\L^{-1}T}\Big),$$
$$Z^1_f(T)=\frac{\L-1}{(\L T^{-1}-1)^2} \hbox{ and } S_f^1=-(\L-1).$$ 
Of course we find that $\chi_c(S_f)=\chi_c(\{ f=c \}\cap \bar B(0,1))=2$, 
$0<c\ll 1$. 

Now let us for instance choose $\{ xy=c,1-x^2-y^2>0\}$, for $0<c\ll 1$, as   a basic semialgebraic formula 
to represent the open Milnor fibre of $f=0$ and let us compute
$\beta([xy=c,1-x^2-y^2>0])$ (rather than $[xy=c,1-x^2-y^2>0]$ itself, since we use regular homeomorphims in our computations).
By definition of the realization $\beta:K_0(BSA_\R)\to  \Z[\frac{1}{2}][u]$, 
we have 
$$\beta([xy=c,1-x^2-y^2>0])=\frac{1}{4}\beta([xy=c,z^2=1-x^2-y^2])$$
$$
-\frac{1}{4}\beta([xy=c,z^2=x^2+y^2-1])+
 \frac{1}{2}\beta([xy=c,1-x^2-y^2\not=0]).$$
Projecting the algebraic set $ \{ xy=c,z^2=1-x^2-y^2\}$ orthogonally to the plane $x=-y$ with 
coordinates $(X=1/\sqrt 2(x-y),z)$ one finds
twice the quadric $z^2+2X^2=1-2c$ that is, up to regular homeomorphism, two circles. A circle having 
 class $u+1$, we have
$$\beta([xy=c,z^2=1-x^2-y^2])=2(u+1). $$ 
Projecting the algebraic set $ \{ xy=c,z^2=x^2+y^2-1\}$ to the plane $x=-y$ with 
coordinates $(X=1/\sqrt 2(x-y),z)$ one finds
twice the hyperbola $2X^2-z^2=1-2c$. Projecting orthogonally again the hyperbola onto one of its asymptotic 
axes we see that this hyperbola has class $u-1$. It gives 
$$\beta([xy=c,z^2=x^2+y^2-1])=2(u-1).$$
Finally the constructible set $\{xy=c,1-x^2-y^2\not=0\}$ being the hyperbola without 
$4$ points, we have
$$\beta([xy=c,1-x^2-y^2>0])=\frac{1}{2}(u+1)-\frac{1}{2}(u-1)+\frac{1}{2}(u-1)-2=
\frac{u-3}{2}. $$
Of course $\chi_c(\chi([xy=c,1-x^2-y^2>0]))=\chi_c(\{f=c\}\cap B(0,1)) =-2$.

The simple semialgebraic formula representing the set theoretic closed Milnor fibre is 
$\{xy=c,1-x^2-y^2\ge0\}$, it has class $\displaystyle \beta([xy=c,1-x^2-y^2>0])+4\beta([\{*\}])=\frac{u+5}{2}$
in $ \Z[\frac{1}{2}][u]$. But although 
$$\chi_c(\chi([xy=c,1-x^2-y^2\ge0]))=\chi_c(S_f^1)=\chi_c(\{f=c\}\cap \bar B(0,1))=2$$ 
as expected from Theorem \ref{Milnor}, 
we observe that 
$$\frac{u+5}{2}=\beta([xy=c,1-x^2-y^2\ge0]) \not=\beta(S_f^1)=-(u-1).$$
As a final consequence, we certainly cannot have 
this equality between $\chi([xy=c,1-x^2-y^2\ge0])$ and $S_f^1 $ at the level of $K_0(\Var_\R)\otimes \Z[\frac{1}{2}]$.

\end{example}

\vskip2cm

\bibliographystyle{amsplain}

\end{document}